\documentclass[a4paper]{amsart}

\usepackage{amsmath,amssymb,amscd}
\usepackage{extarrows}
\usepackage{tikz-cd}

\numberwithin{equation}{section}
\newtheorem{theorem}[equation]{Theorem}

\newtheorem{corollary}[equation]{Corollary}
\newtheorem{lemma}[equation]{Lemma}

\newtheorem{proposition}[equation]{Proposition}

\theoremstyle{remark}
\newtheorem{remark}[equation]{Remark}
\theoremstyle{definition}
\newtheorem{definition}[equation]{Definition}

\newtheorem{example}[equation]{Example}

\newcommand*{\kokoni}[1]{\makebox[0pt][l]{#1}}

\newcounter{proofsec}
\expandafter\let\expandafter\oldproof\csname\string\proof\endcsname

\renewenvironment{proof}[1][\proofname]{\begin{oldproof}[#1]
    \setcounter{proofsec}{-1}}{\end{oldproof}}
\newcommand{\proofsec}{\refstepcounter{proofsec}\noindent\textbf{\theproofsec.}~}

\newcommand{\Alg}{\ccat{Alg}}

\newcommand{\Bimod}{\ccat{Bimod}}
\newcommand{\blank}{(\;)}

\newcommand{\cat}{\mathcal}
\newcommand{\Cat}{\ccat{Cat}}

\newcommand{\ccat}{\mathrm}
\newcommand{\Coalg}{\ccat{Coalg}}

\DeclareMathOperator*{\colim}{colim}

\newcommand{\complete}{\widehatv}

\newcommand{\cotensor}{\Box}

\renewcommand{\dot}{\bullet}

\newcommand{\Dual}{\mathbb{D}}

\renewcommand{\equiv}{\sim}

\newcommand{\equivwith}{\simeq}
\newcommand{\equivto}{\xrightarrow{\equiv}}
\newcommand{\exact}{\mathrm{ex}}

\newcommand{\facemaps}{f}

\DeclareMathOperator{\Fibre}{Fibre}
\newcommand{\field}{\mathbb}

\newcommand{\Z}{\field{Z}}

\newcommand{\from}{\leftarrow}
\newcommand{\Fun}{\ccat{Fun}}

\newcommand{\id}{\mathrm{id}}

\newcommand{\includes}{\supset}

\newcommand{\intersect}{\cap}
\newcommand{\Intersect}{\bigcap}

\newcommand{\into}{\hookrightarrow}

\newcommand{\kore}{\textbf}
\newcommand{\kyokshoka}{}

\newcommand{\leftlocal}{\kyokshoka{\ell}}
\newcommand{\rightlocal}{\kyokshoka{r}}

\newcommand{\longfrom}{\longleftarrow}
\newcommand{\longto}{\longrightarrow}
\newcommand{\Loop}{\Omega}

\newcommand{\Map}{\mathrm{Map}}
\newcommand{\Mod}{\ccat{Mod}}

\newcommand{\naga}{\text{--}}
\newcommand{\naraba}{$\implies$} %

\newcommand{\op}{\mathrm{op}}

\newcommand*{\pairing}[1]{\left\langle#1\right\rangle}

\newcommand{\positive}{+}

\newcommand{\Simp}{\mathbf{\Delta}}

\newcommand{\Simpface}{\Simp_\facemaps}

\DeclareMathOperator{\sk}{sk}

\newcommand{\kocoprod}{\amalg}
\newcommand{\Space}{\ccat{Space}}

\newcommand{\Stable}{\ccat{Sta}}
\newcommand{\sub}{\subset}

\newcommand{\susp}{\Sigma}

\newcommand{\tensor}{\otimes}
\newcommand{\Tensor}{\bigotimes}

\DeclareMathOperator{\Tot}{Tot}

\newcommand*{\keepheight}[2]{#1{#2}\vphantom{{#2}}}

\newcommand{\unit}{\eta}
\newcommand{\counit}{\epsilon}
\newcommand{\unity}{\mathbf{1}}
\newcommand{\widebar}{\overline}

\newcommand*{\widehatv}[1]{\keepheight{\widehat}{#1}}

\newcommand{\zero}{\mathbf{0}}

\begin{document}
\setcounter{section}{-1}
\setcounter{equation}{-1}

\title[Koszul duality for filtered $E_n$-algebras]{Koszul duality
  between $E_n$-algebras and coalgebras in a filtered category}
\author[Matsuoka, Takuo]{Takuo Matsuoka}
\address{Department of Mathematics, Northwestern University,
Evanston, IL,
USA (Not current)}
\email{motogeomtop@gmail.com}
\date{}

\begin{abstract}
We study the Koszul duality between augmented $E_n$-algebras and
augmented $E_n$-coalgebras in a symmetric monoidal stable infinity
$1$-category equipped with a filtration in a suitable sense.
We obtain that the Koszul duality constructions restrict to an
equivalence between augmented algebras and coalgebras which have some
positivity and completeness with respect to the filtration.
We also obtain that the Koszul duality construction is functorial
between carefully constructed generalized Morita categories consisting
of those algebras\slash coalgebras in each dimension.
\end{abstract}

\maketitle 


\section{Introduction}
\label{sec:introduction}
\setcounter{subsubsection}{-1}
\setcounter{equation}{-1}

\subsubsection{}

Let $n$ be a non-negative integer.
The notion of an $E_n$-algebra was first introduced in iterated loop
space theory, in the work of Boardman and Vogt \cite{bv} (including
the case ``$n=\infty$'', which we exclude from our consideration in
this work).
$E_1$-algebra is an associative algebra, and an \emph{$E_n$-algebra}
can be inductively defined as an $E_{n-1}$-algebra with an additional
structure of an associative algebra commuting with the
$E_{n-1}$-structure.
In other words, a structure of an $E_n$-algebra consists of $n$-fold
associative structures and data for compatibility among them.

There is an issue that the notion of an $E_n$-algebra degenerates
(unless $n\le 1$) to that of a commutative algebra in a category whose
higher homotopical structure is degenerate.
Moreover, the kind of theory we aim to establish (the theory of the
Koszul duality) fails in such a setting even for (the case $n=1$ of)
associative algebras.
These issues force us to work in a homotopical setting.
In order to work in such a setting, we use the convenient language of
higher category theory.
(For the main body, note our conventions stated in Section
\ref{sec:terminology-notation}, which do not apply in this
introduction.)
We just remark here that associativity of an algebra in such a
setting means a data for homotopy coherent associativity (which in
particular is a structure rather than a property).

\subsubsection{}

In this paper, we study the Koszul duality between $E_n$-algebras and
$E_n$-coalgebras.
By the Koszul dual of an augmented associative algebra, we mean the
augmented associative coalgebra obtained as the bar construction or
a suitable derived tensor product (see Section
\ref{sec:koszul-complete-algebra}).
For $E_n$-algebras, we simply consider the $n$-fold iteration of this
construction.
See \ref{sec:koszul-duality-e_n} for a review.
(This, essentially well-known, construction is
much simpler than the more general version the author needed in
\cite{poincare}.
The present article is self-contained without the latter.)

Coalgebras are simply algebras in the opposite category, so we
obtain an augmented algebra from an augmented coalgebra by the same
construction.

In some cases, this correspondence between algebras
and coalgebras is an equivalence or close to it.
Results of this form include the \emph{iterated loop space theory} and
the \emph{Verdier duality}.
(See Lurie's book \cite[Section 5.3]{higher-alg} for the relation
between these.)

In other contexts, the correspondence is far from an equivalence.
In a reasonable stable infinity $1$-category with multilinear (and
suitably colimit preserving) symmetric monoidal structure, it instead
happens often, by the formal deformation theory, that the infinity
$1$-category of algebras compares better with the infinity $1$-category of
suitable class of infinitesimal stacks around a point, in a derived
version of suitable geometry.
(See for instance, Francis \cite{tangent}, Lurie \cite{formal}, Hirsch
\cite{hirsch} for some precise such results, even though the idea is
older, especially in the case of the commutative geometry.)

Nevertheless, the important instances of equivanlence between algebras
and coalgebras mentioned above may leave one curious about the
relation between algebras and coalgebras in contexts which are closer
to the latter as well.
In this work, we obtain, with the help of some additional structure
on the ambient symmetric monoidal category, simple classes of
$E_n$-algebras and coalgebras between which the Koszul duality gives
an equivalence of infinity $1$-categories.
See Theorem \ref{thm:equiv-intro} below.

We emphasize that, due to the formal deformation theory just
described, we cannot normally hope, in our context, the Koszul duality
to equate the \emph{whole} infinity $1$-category of algebras with the
coalgebras.
Our context is indeed very different from the contexts in
which one had this equivalence (or almost that).
However, our \emph{result} fits in an analogy with the results in such
contexts.

Our study towards this has also led to a construction of a new,
non-trivial version of the higher Morita category consisting of
certain $E_n$-\emph{co}algebras with bimodule structures for $n$ up to
a chosen value.
Theorem \ref{thm:morita-intro} below shows that the Koszul duality
relates this with the
usual higher Morita category \cite{tft} consisting of algebras.
The latter is interesting for abundance of topological field theories
in it, and the description of the field theories through the
topolocal chiral homology \cite{tft}.
Our theorem in fact leads to an analogous description of the
corresponding theories in the coalgebraic higher Morita category.

\subsubsection{}

The structure we consider on a symmetric monoidal category is a
filtration with respect to which the category becomes complete.
Let us set up our context.

Let $\cat{A}$ be a symmetric monoidal stable infinity $1$-category.
We assume that it has a filtration (Definition \ref{def:filtration},
note the conventions stated in Section \ref{sec:terminology-notation})
which is compatible with the symmetric monoidal structure in a
suitable way.

Primary examples are the category of \emph{filtered objects} in a
reasonable symmetric monoidal stable infinity $1$-category
(Section \ref{sec:examples-of-filtered-categories}), and a symmetric
monoidal stable infinity $1$-category with a compatible \emph{t-structure}
\cite{higher-alg} (satisfying a mild technical condition, see
Definition \ref{def:bounded-limit}, Remark \ref{rem:bounded-limit}).
Another family of examples is given by functor categories
admitting the \emph{Goodwillie calculus} \cite{goodwi}, where the
filtration is given by the degree of excisiveness
(Example \ref{ex:goodwillie-filtration}).

We further assume that $\cat{A}$ is complete with respect to the
filtration in a suitable sense.
The mentioned examples admit completion, and in these examples,
the category $\cat{A}$ we indeed work in is the category of
complete objects in any of the mentioned categories, with completed
symmetric monoidal structure.

These categories satisfy (in particular, Remark
\ref{rem:sound-example}) a few further technical assumptions we need,
which we shall not state here.
(The theorems we state in this introduction will be given references
to their precise formulation in the main body.
In order to understand the formulation there correctly, the reader
should note our conventions stated in Section
\ref{sec:terminology-notation}.)

In such a complete filtered infinity $1$-category $\cat{A}$, any algebra
comes with a natural filtration with respect to which it is complete.
In the mentioned examples, the towers associated to the filtration are
the canonical (or ``defining'') tower, the Postnikov tower, and the
Taylor tower, and the objects we deal with are the limits of the
towers.
We have established the Koszul duality for $E_n$-algebras in $\cat{A}$
which is \emph{positively} filtered.
See Definition \ref{def:e-n-copositive}, as well as comments right
after it on examples of positive augmented algebras.
The corresponding restriction on the filtration of coalgebras is given
by the condition we call \kore{copositivity}
(Definition \ref{def:e-n-copositive}).
Our first main theorem is as follows.

\begin{theorem}[Theorem \ref{thm:associative-koszul-duality}]
\label{thm:equiv-intro}
Let $\cat{A}$ be as above.
Then the constructions of Koszul duals give inverse equivalences
\[
\Alg_{E_n}(\cat{A})_\positive\xlongleftrightarrow{\equiv}\Coalg_{E_n}(\cat{A})_\positive
\]
between the infinity $1$-category of positive augmented $E_n$-algebras and
copositive augmented $E_n$-coalgebras in $\cat{A}$.
\end{theorem}

We have also shown that the Koszul duality further has a
\emph{Morita theoretic} functoriality.

To explain what this is, in \cite{tft}, Lurie has outlined a
generalization for $E_n$-algebras of the ``Morita'' category due to
B\'enabou \cite{benabou}.
By collecting suitable versions of bimodules, one obtains an infinity
($n+1$)-category $\Alg_n(\cat{A})$, in which
\begin{itemize}
\item an object is an $E_n$-algebra in $\cat{A}$,
\item a $1$-morphism is an $E_{n-1}$-algebra in $\cat{A}$ equipped
  with the structure of a suitable kind of bimodule,
\item a $2$-morphisms is an $E_{n-2}$-algebra in $\cat{A}$ equipped
  with the structure of a suitable kind of bimodule,
\end{itemize}
and so on, generalizing the $2$-category of associative algebras and
bimodules.

In order to make the construction of this work, one usually assumes
that the monoidal multiplication functors preserve geometric
realizations variablewise.
However, unless the monoidal multiplication also preserve
totalizations, one cannot have both algebraic and coalgebraic versions
of this in the same way.
We have shown that in the kind of complete filtered category we work
in, the construction works for both positive augmented algebras and
copositive augmented coalgebras at the same time, despite the
mentioned difficulties one would have if one were to include all
algebras and coalgebras.

Let us denote the infinity ($n+1$)-categories we obtain by
$\Alg^\positive_n(\cat{A})$ and $\Coalg^\positive_n(\cat{A})$
respectively.
We have shown the following, second main result of this article.

\begin{theorem}[Theorem \ref{thm:koszul-in-morita}]
\label{thm:morita-intro}
Let $\cat{A}$ be as above.
Then for every $n$, the construction of the Koszul dual define a
symmetric monoidal functor
\[
\blank^!\colon\Alg^\positive_n(\cat{A})\longto\Coalg^\positive_n(\cat{A}).
\]
It is an equivalence with inverse given by the Koszul duality
construction.
\end{theorem}

This has an interesting consequence since any object of the source
category here is $n$-dualizable, just as in the usual non-augmented
case.
The associated $n$-dimensional topological quantum field theory can
then be sent over to a theory in $\Coalg^\positive_n(\cat{A})$.
The field theory in $\Alg^\positive_n(\cat{A})$ has a concrete
description due to Lurie \cite{tft}, in terms of the topological
chiral homology.
One obtains from this an analogously concrete construction of the
theory in $\Coalg^\positive_n(\cat{A})$, by means of the ``compactly
supported'' topological chiral homology.
See Section \ref{sec:morita-structure} and \cite{poincare} for the
details.

A similar result was earlier obtained by Francis \cite{francis}. 

The fact that Lurie's construction indeed gives a topological field
theory in $\Alg^\positive_n(\cat{A})$ (and $\Alg_n(\cat{A})$) can be
seen as a consequence of the gluing property of the topological chiral
homology.
As we explain in \cite{poincare}, the fact that the other
construction, with the compactly supported topological chiral
homology, indeed gives a topological field theory in
$\Coalg^\positive_n(\cat{A})$, can alternatively be
seen as a consequence of a Poincar\'e type duality theorem for the
topological chiral homology.

The Poincar\'e type theorem itself is independent of Theorem
\ref{thm:morita-intro}.
On the other hand, Theorem \ref{thm:morita-intro} required the
non-trivial work of constructing the coalgebraic version of the higher
Morita category,
which is irrelevant to the Poincar\'e type theorem itself.
A version of the Poincar\'e theorem which readily applies to the above
context is treated in \cite{poincare}.
Related results can be found in the work of
Francis \cite{francis}, and Ayala and Francis \cite{ayala-fran}.

\subsection*{Outline}
\textbf{Section \ref{sec:terminology-notation}} is for introducing
conventions which are used throughout the main body.

In \textbf{Sections \ref{sec:filtered}} and
\textbf{\ref{sec:monoidal-filtered}}, we establish basic notions and
facts on symmetric monoidal filtered stable categories.

In \textbf{Section \ref{sec:complete-poincare}}, we develop the theory
of Koszul duality for complete $E_n$-algebras.

\subsection*{Notes on the relation to other articles by the author}
This paper, together with \cite{poincare} and the present author's
paper
\begin{itemize}
\item[{[a]}] \emph{Descent properties of the
    topological chiral homology.}
  arXiv:1409.6944,
\end{itemize}
is based on his Ph.D.~thesis (accepted in April 2014).
The present article is logically independent of either of
\cite{poincare}, [a].

\subsection*{Acknowledgment}
This paper is based on part of my Ph.D.~thesis.
I am particularly grateful to my thesis advisor Kevin Costello for
his
extremely patient guidance and continuous encouragement and support.
Special thanks are due to John Francis for detailed comments and
suggestions on the drafts of my thesis, which were essential for 
many improvements of both the contents and exposition.
Many of those improvements were inherited by this paper.
I am also grateful to Josh Shadlen and Justin Thomas for interesting
conversations, which directly influenced some parts of the present
work.
I am grateful to Owen Gwilliam, Josh, and Yuan Shen for their
continuous encouragement.

\section{Terminology and notations}
\label{sec:terminology-notation}
\setcounter{subsubsection}{-1}
\setcounter{equation}{-1}

\subsubsection{}

By a \kore{$1$-category}, we always mean an \emph{infinity}
$1$-category.
We often call a $1$-category (namely an infinity $1$-category) simply
a \kore{category}.
A category with discrete sets of morphisms (namely, a ``category''
in the more traditional sense) will be called a \emph{discrete}
category.

In fact, all categorical and algebraic terms will be used in
\emph{infinity} ($1$-) categorical sense without further notice.
Namely, categorical terms are used in the sense enriched in the
\emph{infinity} $1$-category of spaces, or equivalently, of infinity
groupoids, and algebraic terms are used freely in the sense
generalized in accordance with the enriched categorical structures.

For example, for an integer $n\ge 1$, by an \emph{$n$-category}, we
mean an \emph{infinity} $n$-category.
We also consider multicategories.
By default, multimaps in our multicategories will form
a \emph{space} with all higher homotopies allowed.
Namely, our ``\emph{multicategories}'' are ``infinity operads'' in the
terminology of Lurie's book \cite{higher-alg}.

\begin{remark}
We usually treat a space relatively to the structure of the standard
(infinity) $1$-category of spaces.
Namely, a ``\emph{space}'' for us is usually no more than an object of
this category.
Without loss of information, we shall freely identify a space in this
sense with its fundamental infinity groupoid, and call it also a
``\emph{groupoid}''.
Exceptions in which the term ``space'' means not necessarily
this, include a ``Euclidean space'', the ``total space'' of a fibre
bundle, etc., in accordance with the common customs.
\end{remark}

\subsubsection{}

If $\cat{C}$ is a category and $x$ is an object of $\cat{C}$,
then we denote by $\cat{C}_{/x}$, the ``\emph{over}'' category, of
objects of
$\cat{C}$ lying over $x$, i.e., equipped with a map to $x$.
We denote the ``\emph{under}'' category for $x$, in other words,
$\bigl((\cat{C}^\op)_{/x}\bigr)^\op$, by $\cat{C}_{x/}$.

More generally, if a category $\cat{D}$ is equipped with a functor to
$\cat{C}$, then we define
$\cat{D}_{/x}:=\cat{D}\times_{\cat{C}}\cat{C}_{/x}$, and similarly for
$\cat{D}_{x/}$.
Note here that $\cat{C}_{/x}$ is mapping to $\cat{C}$ by the functor
which forgets the structure map to $x$.
Note that the notation is abusive in that the name of the functor
$\cat{D}\to\cat{C}$ is dropped from it.
In order to avoid this abuse from causing any confusion, we shall use
this notation only when the functor $\cat{D}\to\cat{C}$ that we are
considering is clear from the context.

\subsubsection{}
By the \kore{lax colimit} of a diagram in the category
$\Cat$ of categories (of a limited size), indexed by a category
$\cat{C}$, we mean the Grothendieck construction.
We choose the variance of the laxness so the lax colimit projects to
$\cat{C}$, to make it an op-fibration over $\cat{C}$, rather
than a fibration over $\cat{C}^\op$.
(In particular, if $\cat{C}=\cat{D}^\op$, so the functor is
contravariant on $\cat{D}$, then the familiar fibred category over
$\cat{D}$ is the \emph{op}-lax colimit over $\cat{C}$ for us.)
Of course, we can choose the variance for lax \emph{limits} compatibly
with this, so our lax colimit generalizes to that in any $2$-category.

\section{Filtered stable category}
\label{sec:filtered}
\setcounter{subsection}{-1}
\setcounter{equation}{-1}

\subsection{Introduction}

In this paper, we consider the Koszul duality in a symmetric monoidal
stable category $\cat{A}$, equipped with a ``\emph{filtration}'' with
respect to which $\cat{A}$ becomes \emph{complete}, or at least can
be completed.
The primary example will be given by the category of complete filtered
objects, which will be reviewed in
Section \ref{sec:examples-of-filtered-categories}.
In fact, the influence to the present work comes from the use of
complete filtered objects in a related context in Costello's
\cite{yangian} (see also the appendix of Costello--Gwilliam
\cite{cg}).
Filtration and completeness are also used in the work of Positselski
on the Koszul duality \cite{positsel}.

Our approach, despite its slight abstractness, has the advantage of
including a few more examples such as the filtration given by a
t-structure, and hopefully of clarifying some logic.
We shall develop these notions in this and the next sections, and then
develop the Koszul duality theory
in such a category, in Section \ref{sec:complete-poincare}.

\subsection{Localization of a stable category}
\setcounter{subsubsection}{-1}

\subsubsection{}

We review some facts we need.

\begin{definition}
Let $\cat{C}$ be a category.

A functor $\cat{C}\to\cat{D}$ is a \kore{left localization}
if it has a fully faithful functor as a right adjoint.

A full subcategory $\cat{D}$ of $\cat{C}$ is a \kore{left
  localization} of $\cat{C}$ if the inclusion functor
$\cat{D}\into\cat{C}$ has a left adjoint.
\end{definition}
\kore{Right} localization is defined similarly, so it is
just left localization in the opposite variance.

We consider the following situation.
Let $\cat{A}$ be a stable category, and let
$\cat{A}_\leftlocal\sub\cat{A}$ be a full subcategory which is a left
localization of $\cat{A}$.
Denote by $\blank_\leftlocal$ the localization functor
$\cat{A}\to\cat{A}_\leftlocal$.
By abuse of notation, we also denote by $\blank_\leftlocal$ the
composite
\[\begin{tikzcd}
\cat{A}\arrow{r}{\blank_\leftlocal}
&\cat{A}_\leftlocal\arrow[hook]{r}
&\cat{A}.
\end{tikzcd}\]

\begin{definition}\label{def:complementary-localization}
A right localization $\cat{A}_\rightlocal$ of $\cat{A}$ is
\kore{complementary} to the left localization $\cat{A}_\leftlocal$ of
$\cat{A}$ as above if for every $X\in\cat{A}_\rightlocal$ and
$Y\in\cat{A}_\leftlocal$, the space $\Map(X,Y)$ is contractible, and
the sequence
\[
\blank_\rightlocal\xrightarrow{\counit}\id\xrightarrow{\unit}\blank_\leftlocal\,\colon\:\cat{A}\longto\cat{A},
\]
where $\blank_\rightlocal$ is the right localization functor
considered as $\cat{A}\to\cat{A}$, and the maps are the counit and the
unit maps for the respective adjunctions, is a fibre sequence (by
the unique null homotopy of the composite $\unit\counit$).
\end{definition}

As a full subcategory of $\cat{A}$, $\cat{A}_\rightlocal$ consists of
objects $X\in\cat{A}$ for which the counit $\counit\colon
X_\rightlocal\to X$ is an equivalence, or equivalently,
$X_\leftlocal\equivwith\zero$.
It follows that given any left localization $\cat{A}_\leftlocal$ of
$\cat{A}$, if it has a complementary right localization, then
the right localization is characterized as the right localization
to the full subcategory of $\cat{A}$ consisting of objects
$X\in\cat{A}$ for which $X_\leftlocal\equivwith\zero$.

Given any right localization, its \kore{complementary} \emph{left}
localization is defined in the opposite way.
It is immediate that if a left localization has a complementary right
localization, then this left localization is left
complementary to its right complement.

\begin{lemma}\label{lem:pointedness}
Let $\cat{A}$ be a stable category, and let $\cat{A}_\leftlocal$,
$\cat{A}_\rightlocal$ be left and right localizations of $\cat{A}$
respectively which are complementary to each other.
Then $\cat{A}_\leftlocal$ is pointed, namely, have zero objects, and
dually for $\cat{A}_\rightlocal$.
\end{lemma}
\begin{remark}
All inclusion and localization functors then preserves the zero objects.
\end{remark}
\begin{proof}
$\zero_{\rightlocal\leftlocal}\equivwith\zero$ implies
$\zero\in\cat{A}_\leftlocal$, which is then a zero object of
$\cat{A}_\leftlocal$.
\end{proof}

\begin{proposition}\label{prop:complementary-localization}
A left localization
$\blank_\leftlocal\colon\cat{A}\to\cat{A}_\leftlocal$ has a
complementary right localization if and only if
\[
\bigl(\Fibre[\unit\colon\id\to\blank_\leftlocal]\bigr)_\leftlocal\:\equivwith\:\zero.
\]
\end{proposition}

\begin{example}\label{ex:complement-exact-localization}
This condition is satisfied if the left localization is exact in the
following sense.
\end{example}

\begin{definition}
A left localization of a stable category $\cat{A}$ is
\kore{exact} if the localization functor
$\blank_\leftlocal\colon\cat{A}\to\cat{A}_\leftlocal$ (and
equivalently, $\blank_\leftlocal\colon\cat{A}\to\cat{A}$) preserves
finite limits.
\end{definition}

A right localization is \kore{exact} if the localization functor is
exact.

\begin{proof}[Proof of
  Proposition~\ref{prop:complementary-localization}]
\textbf{Necessity} follows from the remark for
Definition \ref{def:complementary-localization}.

For \textbf{sufficiency}, define $\cat{A}_\rightlocal$ as the full
subcategory of $\cat{A}$ consisting of objects $X\in\cat{A}$ for which
$X_\leftlocal\equivwith\zero$.
Then the functor
$\blank_\rightlocal:=\Fibre[\unit\colon\id\to\blank_\leftlocal]\colon\cat{A}\to\cat{A}$
lands in $\cat{A}_\rightlocal$.
Denote the resulting functor $\cat{A}\to\cat{A}_\rightlocal$ also by
$\blank_\rightlocal$.

It will then follow that $\blank_\rightlocal$ is a right adjoint of
the inclusion $\cat{A}_\rightlocal\into\cat{A}$,
with counit the canonical map $\blank_\rightlocal\to\id$.
Indeed, the defining fibre sequence for $X_\rightlocal$ gives for any $Y$,
the fibre sequence
\[
\Map(Y,X_\rightlocal)\longrightarrow\Map(Y,X)\longrightarrow\Map(Y,X_\leftlocal),
\]
but since $X_\leftlocal\in\cat{A}_\leftlocal$, we have
$\Map(Y,X_\leftlocal)=\Map(Y_\leftlocal,X_\leftlocal)$.
Now, if $Y\in\cat{A}_\rightlocal$, then this space is contractible, so
we obtain that the map $\Map(Y,X_\rightlocal)\to\Map(Y,X)$ is an
equivalence, as was to be shown.

Thus we have obtained a right localization
$\blank_\rightlocal\colon\cat{A}\to\cat{A}_\rightlocal$, and this came
as complementary to the left localization we started with.
\end{proof}

In the case of exact localization, the localizations will further be
stable, as will be proved in
Proposition~\ref{prop:exact-localization-is-stable} below.

\subsubsection{}

\begin{lemma}\label{lem:local-equivalence}
Let $\cat{A}$ be a stable category with complementary left and right
localizations $\blank_\leftlocal\colon\cat{A}\to\cat{A}_\leftlocal$
and $\blank_\rightlocal\colon\cat{A}\to\cat{A}_\rightlocal$
respectively.
Then, for a cofibre sequence
\[
W\longto X\longto Y
\]
in $\cat{A}$, if $W$ belongs to the full subcategory
$\cat{A}_\rightlocal$ of $\cat{A}$, then the localized
map $X_\leftlocal\to Y_\leftlocal$ is an equivalence.
\end{lemma}
\begin{proof}
$W$ belongs to $\cat{A}_\rightlocal$ if and only if
$W_\leftlocal\equivwith\zero$.

By applying the localization functor
$\blank_\leftlocal\colon\cat{A}\to\cat{A}_{\leftlocal}$ to the given
cofibre sequence, we obtain a cofibre sequence in
$\cat{A}_\leftlocal$.
If $W_\leftlocal\equivwith\zero$, then the map
$X_\leftlocal\to Y_\leftlocal$ in the sequence is an equivalence.
\end{proof}

\begin{corollary}\label{cor:closure-under-extension}
In the situation of Lemma \ref{lem:local-equivalence}, if $Y$ also
belongs to $\cat{A}_\rightlocal$, then $X$ belongs to
$\cat{A}_\rightlocal$.
\end{corollary}

\begin{corollary}\label{cor:characterizing-localizations}
In the situation of Lemma \ref{lem:local-equivalence}, if
$Y\in\cat{A}_\leftlocal$, then the canonical map $W\to X_\rightlocal$
and $X_\leftlocal\to Y$ are equivalences, so the fibre sequence is
canonically equivalent to the canonical fibre sequence
\[
X_\rightlocal\longto X\longto X_\leftlocal.
\]
\end{corollary}
\begin{proof}
The equivalences of objects is immediate from
Lemma \ref{lem:local-equivalence}.
The fibre sequences will then be canonically the same since the
null-homotopy of the composite is unique.
\end{proof}

\subsubsection{}

In a situation where we have left and right localizations
complementary to each other, we shall be particularly interested in
how
the localizations interact with limits (and colimits) in our stable
category.
We have seen in Lemma \ref{lem:pointedness} that localizations contain
$\zero\in\cat{A}$.
More generally, we have the following.

\begin{lemma}\label{lem:left-localization-is-closed-under-limits}
If a left localization $\cat{A}_\leftlocal$ has a complementary right
localization, then $\cat{A}_\leftlocal$ is closed in $\cat{A}$ under
any limit which exists in $\cat{A}$.
\end{lemma}
\begin{proof}
This follows since $\cat{A}_\leftlocal$ is the full subcategory of
$\cat{A}$ consisting of $X\in\cat{A}$ for which
$X_\rightlocal\equivwith\zero$, and since the functor
$\blank_\rightlocal\colon\cat{A}\to\cat{A}_\rightlocal$ is a right
adjoint, and hence preserves any limit.
\end{proof}

Note also that the limit taken in $\cat{A}$ of a diagram lying in the
full subcategory $\cat{A}_\leftlocal$ (which in fact belongs to
$\cat{A}_\leftlocal$, according to the above) will be a limit in
$\cat{A}_\leftlocal$ of the diagram.
On the other hand, since the inclusion
$\cat{A}_\leftlocal\into\cat{A}$ preserves limits, if a limit of a
diagram $\cat{A}_\leftlocal$ exists in the category
$\cat{A}_\leftlocal$, then it also will be a limit in $\cat{A}$.

\begin{corollary}\label{cor:left-localization-is-closed-under-sum}
If a left localization $\cat{A}_\leftlocal$ has a complementary right
localization, then $\cat{A}_\leftlocal$ is closed in $\cat{A}$ under
the finite coproduct in $\cat{A}$.
\end{corollary}
\begin{proof}
Let $X$, $Y$ be object of $\cat{A}$ which belong to
$\cat{A}_\leftlocal$.
Then the coproduct $X\kocoprod Y$ in $\cat{A}$ is equivalent to the
product $X\times Y$ in $\cat{A}$, which belongs to
$\cat{A}_\leftlocal$ by
Lemma~\ref{lem:left-localization-is-closed-under-limits}.
\end{proof}

\subsubsection{}

In the next proposition, we assume given complementary left and right
localizations $\blank_\leftlocal\colon\cat{A}\to\cat{A}_\leftlocal$
and $\blank_\rightlocal\colon\cat{A}\to\cat{A}_\rightlocal$
respectively, of a stable category $\cat{A}$.

In this situation, we assume given classes of diagrams
$\cat{D}$, $\cat{D}_\leftlocal$, $\cat{D}_\rightlocal$,
in $\cat{A}$, in $\cat{A}_\leftlocal$, and in
$\cat{A}_\rightlocal$ respectively, and consider limits of diagrams
belonging to any of these classes.
For example, we may be considering all finite limits in $\cat{A}$,
$\cat{A}_\leftlocal$, or $\cat{A}_\rightlocal$.
Alternatively, we may be considering sequential limits.
We may also be considering looping of objects.

We require that all inclusion and localization functors between these
categories take a diagram in the specified class in the source to one
in the specified class in the target.
In fact, from this requirement, it is immediate that the class
$\cat{D}$ determines the other classes.
Namely, $\cat{D}_\leftlocal$ is the class of diagrams which belong
to $\cat{D}$ when considered as diagrams in $\cat{A}$, and similarly
for $\cat{D}_\rightlocal$.
One can start from any class of diagrams $\cat{D}$ in $\cat{A}$ which
is closed under application of endofunctors
$\blank_\leftlocal$ and $\blank_\rightlocal\colon\cat{A}\to\cat{A}$, to
have all three classes satisfying our requirements.

In this situation, if $\cat{A}$ has limits of all diagrams in the
class $\cat{D}$, then $\cat{A}_\rightlocal$ has limits of all diagrams
in the class $\cat{D}_\rightlocal$, and by
Lemma \ref{lem:left-localization-is-closed-under-limits},
$\cat{A}_\leftlocal$ has limits of all diagrams in the class
$\cat{D}_\leftlocal$.

\begin{proposition}\label{prop:limit-in-localization}
Let $\cat{A}$, $\cat{A}_\leftlocal$, $\cat{A}_\rightlocal$, and
classes of diagrams $\cat{D}$ in $\cat{A}$, $\cat{D}_\leftlocal$ in
$\cat{A}_\leftlocal$, $\cat{D}_\rightlocal$ in $\cat{A}_\rightlocal$ be
as above.
Assume that $\cat{A}$ has limits of all
diagrams in the class $\cat{D}$ (see above).

Then the following are equivalent.
\begin{enumerate}\setcounter{enumi}{-1}
\item\label{item:left-localization-preserving-limits}
  The left localization functor
  $\blank_\leftlocal\colon\cat{A}\to\cat{A}_\leftlocal$ takes limits
  of diagrams belonging to $\cat{D}$, to corresponding limits in
  $\cat{A}_\leftlocal$.
\item The functor $\blank_\leftlocal$ considered as
  $\cat{A}\to\cat{A}$, takes limits of diagrams belonging to
  $\cat{D}$, to corresponding limits in $\cat{A}$.
\item\label{item:localization-preserving-limits}
  The right localization functor
  $\blank_\rightlocal\colon\cat{A}\to\cat{A}$ takes limits of diagrams
  belonging to $\cat{D}$, to corresponding limits in $\cat{A}$.
\item\label{item:closed-under-limit}
  Given a diagram in $\cat{A}_\rightlocal$, belonging to
  $\cat{D}_\rightlocal$ (equivalently, a diagram in $\cat{A}$ which
  belongs to $\cat{D}$, and lands in $\cat{A}_\rightlocal$), its limit
  taken in $\cat{A}$, belongs to $\cat{A}_\rightlocal$.
\item\label{item:inclusion-preserving-limits} The inclusion
  $\cat{A}_\rightlocal\into\cat{A}$ takes limits of diagrams belonging
  to $\cat{D}_\rightlocal$, to corresponding limits in $\cat{A}$.
\end{enumerate}
\end{proposition}
\begin{proof}
It is relatively simple to see that the first three are equivalent to
each other.
It is also easy to see that
(\ref{item:localization-preserving-limits})\naraba
(\ref{item:closed-under-limit})\naraba
(\ref{item:inclusion-preserving-limits})\naraba
(\ref{item:localization-preserving-limits}).
\end{proof}

\begin{proposition}\label{prop:exact-localization-is-stable}
Let $\cat{A}$ be a stable category, and let
$\blank_\leftlocal\colon\cat{A}\to\cat{A}_\leftlocal$ be an exact
left localization of $\cat{A}$.
Then the category $\cat{A}_\leftlocal$ is stable, the inclusion
functor $\cat{A}_\leftlocal\into\cat{A}$ is also exact, and the
complementary right localization (see
Example~\ref{ex:complement-exact-localization}) is also exact.
\end{proposition}

\begin{proof}
We apply Proposition~\ref{prop:limit-in-localization} for the class
of all finite limits.
The assumption is the
condition~\eqref{item:left-localization-preserving-limits}.
The condition~\eqref{item:closed-under-limit} then states that the
right localization $\cat{A}_\rightlocal$ is closed as a full
subcategory of
$\cat{A}$ under the formation of finite limits in $\cat{A}$.

It follows from the opposite case of
Lemma~\ref{lem:left-localization-is-closed-under-limits} that
$\cat{A}_\rightlocal$ is also closed in $\cat{A}$ under the formation
of colimits.
It then follows that the pointed (by Lemma~\ref{lem:pointedness})
category $\cat{A}_\rightlocal$ is stable since Cartesian and
coCartesian squares coincide in $\cat{A}_\rightlocal$ since those
squares are Cartesian or coCartesian in $\cat{A}$.

It then follows from a result of
Lurie~\cite[Proposition~1.1.4.1]{higher-alg} that the right
localization functor
$\blank_\rightlocal\colon\cat{A}\to\cat{A}_\rightlocal$ is exact
since it is between stable categories and preserves all limits.

We finally obtain the result by changing the variance in the
discussions up to here.
\end{proof}

\subsection{Filtration of a stable category}
\label{sec:filtration-of-stable-category}
\setcounter{subsubsection}{-1}

\subsubsection{}

\begin{definition}\label{def:filtration}
A \kore{filtration} of a stable category $\cat{A}$ is a sequence of
full subcategories
\[
\cat{A}\:\includes\:\cdots\:\includes\:\cat{A}_{\ge r}\:\includes\:\cat{A}_{\ge r+1}\:\includes\:\cdots
\]
indexed by integers, each of which is the inclusion of a right
localization which has a complementary left localization, denoted by
$\blank^{<r}\colon\cat{A}\to\cat{A}^{<r}$.

A \kore{filtered} stable category is a stable category which is
equipped with a filtration.
\end{definition}

In particular, associated to a filtered stable category $\cat{A}$, we
have a sequence
\[
\cat{A}\longrightarrow\cdots\longrightarrow\cat{A}^{<r+1}\longrightarrow\cat{A}^{<r}\longrightarrow\cdots.
\]
of left localization functors.
We would like to think of this as the tower associated to the
filtration.

$\cat{A}_{\ge r}$ can be considered as the
\emph{pieces} for the filtration.
$\cat{A}_{\ge r}$ is the full subcategory of $\cat{A}$
formed by objects $X\in\cat{A}$ for which $X^{<r}\equivwith\zero$.
We denote the right localization by
$\blank_{\ge r}\colon\cat{A}\to\cat{A}_{\ge r}$.
Then the sequence $\blank_{\ge r}\to\id\to\blank^{<r}$ of
functors $\cat{A}\to\cat{A}$, equipped with the unique null homotopy
of the composite $\blank_{\ge r}\to\blank^{<r}$, is a fibre sequence.

An important example will be discussed in
Section \ref{sec:examples-of-filtered-categories}.
Here are a few examples.

\begin{example}\label{ex:t-structure}
If $\cat{A}$ is a stable category, then any t-structure
\cite{higher-alg} on $\cat{A}$ gives a filtration.
In fact, a t-structure can be characterized as a filtration
satisfying a simple condition.
See Example~\ref{ex:bounded-loop} and
Remark~\ref{rem:sound-example}.
\end{example}

\begin{example}\label{ex:goodwillie-filtration}
Let $\cat{A}$ be the functor category into a stable category, and
assume it admits some version of the Goodwillie calculus
\cite{goodwi}.
Then it has a filtration in which $\cat{A}^{<r}$ is the full
subcategory consisting of ($r-1$)-excisive functors.
The left localization $\cat{A}\to\cat{A}^{<r}$ is given by the
universal ($r-1$)-excisive approximation of functors, which is exact
as follows e.g., from the construction.
\end{example}

\begin{remark}\label{rem:dual-filtration}
The notion of a filtration on a stable category is self-dual in the
following sense.
Namely, if a stable category $\cat{A}$ is given a filtration, then
$\cat{B}:=\cat{A}^\op$ has a filtration given by $\cat{B}_{\ge
  r}:=(\cat{A}^{\le-r})^\op$, where $\cat{A}^{\le s}:=\cat{A}^{<s+1}$.
\end{remark}

Therefore, all notions and statements we formulate will have dual
versions, which we shall speak about freely without further notices.

\subsubsection{}

Let $r$, $s$ be integers such that $r\le s$.
Then $(X_{\ge r})^{<s}$ belongs to
\[
\cat{A}_{\ge r}^{<s}\::=\:\cat{A}_{\ge r}\intersect\cat{A}^{<s}
\]
since
\[
((X_{\ge r})^{<s})^{<r}=(X_{\ge r})^{<r}\equivwith\zero,
\]
and so does $(X^{<s})_{\ge r}$.

We would like to compare these objects.

We have a commutative diagram
\begin{equation}\label{eq:diamond-commutes}
\begin{tikzcd}
{}&X_{\ge r}\arrow{dl}\arrow{d}\arrow{dr}\\
(X_{\ge r})^{<s}\arrow{dr}&X\arrow{d}&(X^{<s})_{\ge r}\arrow{dl}\\
&X^{<s}\kokoni{,}
\end{tikzcd}
\end{equation}
so the universal property of the map $X_{\ge r}\to(X_{\ge r})^{<s}$
implies that there is a unique pair consisting of a map $(X_{\ge
  r})^{<s}\to(X^{<s})_{\ge r}$ and a homotopy making the upper
triangle of
\begin{equation}\label{eq:extension-lift}
\begin{tikzcd}
{}&X_{\ge r}\arrow{dl}\arrow{dr}\\
(X_{\ge r})^{<s}\arrow{dr}\arrow{rr}&&(X^{<s})_{\ge r}\arrow{dl}\\
&X^{<s}
\end{tikzcd}
\end{equation}
commute.

Moreover, again by the universal property of the map $X_{\ge
  r}\to(X_{\ge r})^{<s}$, there is a unique pair consisting of
\begin{itemize}
\item a homotopy filling the lower triangle, and
\item a higher homotopy between the homotopy filling the diamond in
  the diagram \eqref{eq:diamond-commutes}, and the homotopy obtained
  by pasting the homotopies in the diagram \eqref{eq:extension-lift}.
\end{itemize}

In other words, there is a unique quadruple consisting of
\begin{enumerate}
\setcounter{enumi}{-1}
\item\label{item:comparison-map}
  a map $(X_{\ge r})^{<s}\to(X^{<s})_{\ge r}$
\item\label{item:upper-homotopy}
  a homotopy filling the upper triangle of \eqref{eq:extension-lift}
\item\label{item:lower-homotopy}
  a homotopy filling the lower triangle of \eqref{eq:extension-lift}
\item\label{item:high-homotopy} a higher homotopy between the homotopy
  filling the diamond in the diagram \eqref{eq:diamond-commutes}, and
  the homotopy obtained by pasting the homotopies in the diagram
  \eqref{eq:extension-lift}.
\end{enumerate}

Moreover, by the universal property of the map $(X^{<s})_{\ge r}\to
X^{<s}$, the pair given by (\ref{item:comparison-map}) and
(\ref{item:lower-homotopy}) above, must be the unique pair of this
form.

It follows that for a map $(X_{\ge r})^{<s}\to(X^{<s})_{\ge r}$ the
following data (in particular, existence of the data) are equivalent
to each other.

\begin{itemize}
\item (\ref{item:upper-homotopy}) above
\item (\ref{item:lower-homotopy}) above
\item Extension to a quadruple above.
\end{itemize}

\begin{lemma}\label{lem:cutting-both-sides}
Let $r$, $s$ be integers such that $r\le s$.
Then a map $(X_{\ge r})^{<s}\to(X^{<s})_{\ge r}$ which can be equipped
with the equivalent data above, is an equivalence.
\end{lemma}
\begin{proof}
By looking at the cofibre of the map (drawn vertically) of fibre
sequences
\[\begin{tikzcd}
X_{\ge s}\arrow{d}\arrow{r}{=}&X_{\ge
  s}\arrow{d}\arrow{r}&\zero\arrow{d}\\
X_{\ge r}\arrow{r}&X\arrow{r}&X^{<r}\kokoni{,}
\end{tikzcd}\]
we obtain a fibre sequence
\[
(X_{\ge r})^{<s}\longto X^{<s}\longto X^{<r}.
\]
The map $X^{<s}\to X^{<r}$ here is a map under $X$, so can be
identified with the canonical map $X^{<s}\to(X^{<s})^{<r}$.
Therefore, its fibre $(X_{\ge r})^{<s}$ is equivalent to
$(X^{<s})_{\ge r}$ by a map over $X^{<s}$.
\end{proof}

\begin{definition}
Let $r$, $s$ be integers.
Then we denote the canonically equivalent objects $(X_{\ge
  r})^{<s}=(X^{<s})_{\ge r}$ by $X^{<s}_{\ge r}$.
This belongs to $\cat{A}_{\ge r}^{<s}$.
\end{definition}

\subsection{Completion}
\setcounter{subsubsection}{-1}

\subsubsection{}
 
Let $\cat{A}$ be a filtered stable category.
Then define
\[
\cat{A}_{\ge\infty}\:
:=\:\lim_r\cat{A}_{\ge r}=\Intersect_r\cat{A}_{\ge r},
\]
the intersection taken in $\cat{A}$.
We would like to investigate the sequence.
\[
\cat{A}_{\ge\infty}\longto\cat{A}\longto\lim_r\cat{A}^{<r}
\]
obtained as the limit of the sequence
\[
\cat{A}_{\ge r}\longto\cat{A}\longto\cat{A}^{<r}.
\]

Let us denote by $\tau$ the functor $\cat{A}\to\lim_r\cat{A}^{<r}$
here.
If $\cat{A}$ is closed under the sequential limit, then this has a
right adjoint, which we shall denote by $\lim$.
For an object $X=(X_r)_r$ of $\lim_r\cat{A}^{<r}$, it is given by
\[
\lim X\:=\:\lim_rX_r,
\]
where the limit on the right hand side is taken in $\cat{A}$.

\begin{definition}
Let $\cat{A}$ be a filtered stable category which is closed under the
sequential limit.
Then we denote $\lim\tau X$ by $\complete{X}$.
We say that $X$ is \kore{complete} if the unit map $\eta\colon
X\to\lim\tau X=\complete{X}$ for the adjunction is an equivalence.

We denote by $\complete{\cat{A}}$ the full subcategory of $\cat{A}$
consisting of complete objects.
\end{definition}

\begin{example}\label{ex:bouded-above-are-complete}
For every $r$, $\cat{A}^{<r}\sub\complete{\cat{A}}$ in $\cat{A}$.
\end{example}

We compromise with the following definition, which may be more
restrictive than it should be.
\begin{definition}\label{def:completion}
Let $\cat{A}$ be a filtered stable category.
Then we say that $\complete{\cat{A}}$ is the \kore{completion} of
$\cat{A}$ if the following conditions are satisfied.
\begin{enumerate}\setcounter{enumi}{-1}
\item $\cat{A}$ is closed under the sequential limit, so we have
  $\complete{\cat{A}}$ defined.
\item\label{item:sequential-limit} The functor $\complete\blank\colon\cat{A}\to\cat{A}$ preserves
  sequential limits.
\item\label{item:idempotence} $\complete\blank$ lands in
  $\complete{\cat{A}}$.
\item\label{item:localization} The map
  $\eta\colon\id\to\complete\blank$ makes $\complete\blank$ a
  left localization for the full subcategory $\complete{\cat{A}}$.
\end{enumerate}
If $\cat{A}$ has $\complete{\cat{A}}$ as its completion in this sense,
then we call the localization functor the \kore{completion} functor.
In this case, we call $\eta$ the \kore{completion} map.

We say that $\cat{A}$ is \kore{complete} if it is closed under the
sequential limit, and $\complete{\cat{A}}$ is
the whole of $\cat{A}$, namely, if every object of $\cat{A}$ is
complete.
\end{definition}

\begin{remark}
The conditions (\ref{item:idempotence}) and (\ref{item:localization})
follows if $\tau\lim\tau\equivwith\tau$ by the canonical map(s).
This is also necessary since for every $r$,
Example \ref{ex:bouded-above-are-complete} will imply that the map
$\eta^{<r}\colon X^{<r}\to\complete{X}^{<r}$ is an equivalence for
every $X$.
\end{remark}

\subsubsection{}
For the rest of our discussion of completion, we assume that any
filtered stable category which we consider is closed under the
sequential limit.

\subsubsection{}

The following is a part of the motivation for
Definition \ref{def:completion}.
\begin{lemma}
If $\complete{\cat{A}}$ is the completion of $\cat{A}$, then the
sequential limits exists in $\complete{\cat{A}}$, and the completion
functor preserves sequential limits.
\end{lemma}

The following gives a sufficient condition for $\complete{\cat{A}}$ to
be the completion of $\cat{A}$.
\begin{lemma}\label{triangles-for-completion}
If $\tau$ preserves sequential limits, then $\complete{\cat{A}}$ is
the completion of $\cat{A}$.
\end{lemma}
\begin{proof}
The condition (\ref{item:sequential-limit}) of
Definition (\ref{def:completion}) is automatic.

To prove the other conditions, it suffices to prove that
$\tau\lim\tau\equivwith\tau$ by the canonical map(s).
Let $X$ be an object of $\cat{A}$.
Then it suffices to prove that for the unit map $\eta\colon
X\to\lim\tau X$, the map $\eta^{<r}$ is an equivalence for every $r$.
By Lemma \ref{lem:local-equivalence}, it suffices to prove that the
fibre $\lim_sX_{\ge s}$ of $\eta$ belongs to $\cat{A}_{\ge\infty}$.

We have $\tau\lim_sX_{\ge s}=\lim_s\tau X_{\ge s}$, so it suffices to
show that this limit is $\zero$.
However, the limit over $s$ of the $r$-th object of $\tau X_{\ge
  s}$ is $\lim_sX_{\ge s}^{<r}\equivwith\zero$ in $\cat{A}^{<r}$, and coincides
with the $r$-th object of $\zero\in\lim_r\cat{A}^{<r}$.
It follows that this $\zero$ is indeed the limit $\lim_s\tau X_{\ge
  s}$.
\end{proof}

\begin{lemma}\label{lem:completion-is-right-localization}
Let $\cat{A}$ be a filtered stable category.
If the functor $\complete\blank\colon\cat{A}\to\cat{A}$ preserves
sequential limits, then $\lim\tau\lim\equivwith\lim$ by the canonical
map(s).
In particular, if $\complete{\cat{A}}$ is the completion of $\cat{A}$,
then $\lim$ lands in $\complete{\cat{A}}$, and will make
$\complete{\cat{A}}$ a right localization of $\lim_r\cat{A}^{<r}$.
\end{lemma}
\begin{proof}
Let $X=(X_r)_r$ be an object of $\lim_r\cat{A}^{<r}$.
Then
\[
\lim\tau\lim X=\operatorname*{\complete\lim}_rX_r
=\lim_r\complete{X}_r=\lim_rX_r=\lim X.
\]
\end{proof}

\subsubsection{}

It would be natural to ask whether completion has a complementary
right localization.
Let us first give a characterization of objects with vanishing
completion.
\begin{lemma}\label{lem:identifying-complement}
Let $\cat{A}$ be a filtered stable category with $\complete{\cat{A}}$
its completion.
Then the completion of an object $X$ of $\cat{A}$ vanishes if and only
if $X$ belongs to $\cat{A}_{\ge\infty}$.
\end{lemma}
\begin{proof}
$X$ belongs to $\cat{A}_{\ge\infty}$ if and only if $\tau
X\equivwith\zero$.
The result then follows from
Lemma \ref{lem:completion-is-right-localization}.
Indeed, $\tau X$ is contained in the full subcategory
of $\lim_r\cat{A}^{<r}$ which by
Lemma~\ref{lem:completion-is-right-localization}, is a right
localization, and is identified with $\complete{\cat{A}}$.
Therefore, $\tau X\equivwith\zero$ in $\lim_r\cat{A}^{<r}$ if and only
if it is so in this full subcategory of $\lim_r\cat{A}^{<r}$.
However, the object of $\complete{\cat{A}}$ corresponding to $\tau X$
under the identification by
Lemma~\ref{lem:completion-is-right-localization}, is
$\complete{X}\in\complete{\cat{A}}$.
\end{proof}

\begin{lemma}\label{lem:vanishing-completion}
Let $\cat{A}$ be a filtered stable category with $\complete{\cat{A}}$
its completion.
Suppose given an inverse system
\[
\cdots\longfrom X_i\longfrom X_{i+1}\longfrom\cdots
\]
in $\cat{A}$, and suppose there is a sequence
$(r_i)_i$ of integers, tending to $\infty$ as $i\to\infty$,
such that $X_i$ belongs to $\cat{A}_{\ge r_i}$ for every $i$.

Then $\lim_iX_i$ belongs to $\cat{A}_{\ge\infty}$.
\end{lemma}
\begin{proof}
From the previous lemma, it suffices to prove that its completion
vanishes.
However,
\[
\operatorname*{\complete\lim}_iX_i=\lim_i\complete{X}_i=\lim_r\lim_iX_i^{<r}\equivwith\lim_r\zero=\zero.
\]
\end{proof}

\begin{proposition}\label{prop:complement-of-completion}
Let $\cat{A}$ be a filtered stable category with $\complete{\cat{A}}$
its completion.
Then the full subcategory $\cat{A}_{\ge\infty}$ of $\cat{A}$ is a
right localization complementary to the left localization
$\complete{\cat{A}}$.
\end{proposition}
\begin{proof}
It suffices to show that completion has a complementary right
localization, since the right localization will then be identified
with $\cat{A}_{\ge\infty}$ by Lemma \ref{lem:identifying-complement}.
Existence of the complement follows from
Proposition \ref{prop:complementary-localization} and
Lemma \ref{lem:vanishing-completion} since the fibre of the completion
map is $\lim_rX_{\ge r}$.
\end{proof}

\begin{corollary}\label{cor:complete-limit}
Let $\cat{A}$ be a filtered stable category with $\complete{\cat{A}}$
its completion.
Then a limit of complete objects is complete.
In particular, a limit of bounded above objects is complete.
\end{corollary}
\begin{proof}
This follows from Proposition~\ref{prop:complement-of-completion} and
Lemma \ref{lem:left-localization-is-closed-under-limits}.
\end{proof}

\begin{corollary}\label{cor:equivalence-of-limits-after-completion}
Let $\cat{A}$ be a filtered stable category with $\complete{\cat{A}}$
its completion.

Suppose given a map of inverse systems
\[\begin{tikzcd}
\cdots&X_i\arrow{l}\arrow{d}[swap]{f_i}&X_{i+1}\arrow{l}\arrow{d}{f_{i+1}}&\cdots\arrow{l}\\
\cdots&Y_i\arrow{l}&Y_{i+1}\arrow{l}&\cdots\arrow{l}
\end{tikzcd}\]
in $\cat{A}$, and suppose there is a sequence
$(r_i)_i$ of integers, tending to $\infty$ as $i\to\infty$,
such that the fibre of $f_i$ belongs to $\cat{A}_{\ge r_i}$ for every
$i$.

Then the map $\lim_if_i\colon\lim_iX_i\to\lim_iY_i$ is an equivalence
after completion.
\end{corollary}
\begin{proof}
This follows from Lemma~\ref{lem:vanishing-completion},
Proposition~\ref{prop:complement-of-completion} and Lemma
\ref{lem:local-equivalence}.
\end{proof}

\subsubsection{}

In practice, it may not be clear when $\tau$ preserves
sequential limits, since limits in $\lim_r\cat{A}^{<r}$ is not always
objectwise.
The following condition will lead to the same conclusions on the
completion, but involves only the sequential limits in $\cat{A}$.
\begin{definition}\label{def:bounded-limit}
Let $\cat{A}$ be a filtered stable category which is closed under the
sequential limit.
Then we say that \kore{sequential limits are uniformly bounded} in
$\cat{A}$ if there exists an integer $d$ such that for every
integer $r$, and for every inverse sequence in the full subcategory
$\cat{A}_{\ge r}$ of $\cat{A}$, the limit of the sequence taken in
$\cat{A}$, belongs to $\cat{A}_{\ge r+d}$.
We refer to such $d$ as a \kore{uniform lower bound} for sequential
limits in $\cat{A}$.
\end{definition}

\begin{remark}\label{rem:bounded-limit}
$\cat{A}$ is assumed to have finite limits and sequential
limits, so it has countable products at least, and if sequential
limits are uniformly bounded, then so are countable products in the
similar sense.
In the case where the filtration is given by a t-structure, if
countable products in $\cat{A}$ are uniformly bounded below by $b$,
then the familiar computation of a sequential limit in terms of
countable products by Milnor shows that sequential limits will be
bounded by $b-1$.

In the case of Goodwillie's filtration
(Example \ref{ex:goodwillie-filtration}), sequential limits are
bounded below by $0$ assuming that the object-wise sequential limits
exist.
\end{remark}

However, it turns out that in order to prove that $\complete{\cat{A}}$
is the completion of $\cat{A}$ in this case, one necessarily proves
that the functor $\tau$ preserves limits as well.
Namely, we have the following two lemmas.

\begin{lemma}\label{lem:tau-localizes}
Let $\cat{A}$ be a filtered stable category with uniformly bounded
sequential limits.
Then $\tau$ is a left localization.
In other words, the functor $\lim\colon\lim_r\cat{A}^{<r}\to\cat{A}$
lands in $\complete{\cat{A}}$, and induces an equivalence
$\lim_r\cat{A}^{<r}\equivto\complete{\cat{A}}$.
\end{lemma}

\begin{lemma}
In the case $\tau$ is a left localization functor,
$\tau$ preserves sequential limits if and only if
$\complete\blank\colon\cat{A}\to\cat{A}$ preserves sequential limits.
\end{lemma}
\begin{proof}[Proof assuming Lemma~\ref{lem:tau-localizes}]
Through the identification of $\lim_r\cat{A}^{<r}$ with
$\complete{\cat{A}}$ by the equivalence  $\lim$, $\tau$ gets
identified with $\complete\blank\colon\cat{A}\to\complete{\cat{A}}$.
\end{proof}

\begin{proof}[Proof of Lemma~\ref{lem:tau-localizes}]
It suffices to prove that the counit $\varepsilon\colon\tau\lim\to\id$
of the adjunction is an equivalence.

Let $X=(X_r)_r$ be an object of $\lim_r\cat{A}^{<r}$.
Then the counit for the adjunction is given by
\[
(\lim_sX_s)^{<r}\longrightarrow X_r^{<r}=X_r
\]
for each $r$.

Let $d\le 0$ be a uniform lower bound for sequential limits.
We can apply
Lemma \ref{lem:local-equivalence} to the fibre sequence
\[
\lim_s(X_s)_{\ge r-d}\longto\lim_sX_s\longto\lim_s(X_s)^{<r-d},
\]
where the fibre belongs to $\cat{A}_{\ge r}$, and the cofibre is
$X_{r-d}$.
We get that that the induced map $(\lim X)^{<r}\to
X_{r-d}^{<r}=X_r$ is an equivalence.
\end{proof}

\begin{lemma}\label{lem:preservation-of-bounded-sequential-limits}
Let $\cat{A}$ be a filtered stable category with uniformly bounded
sequential limits.
Then $\complete\blank\colon\cat{A}\to\cat{A}$ preserves sequential
limits.
\end{lemma}

\begin{proof}
Let
\[
\cdots\longfrom X_i\longfrom X_{i+1}\longfrom\cdots
\]
be a sequence in $\cat{A}$.
Then
\[
(\lim_iX_i)^{<r}=(\lim_iX_i^{<r-d})^{<r}.
\]
The limit of this as $r\to\infty$ can then be computed as
$\lim_s\lim_r(\lim_iX_i^{<s})^{<r}$, but $\lim_iX_i^{<s}$ belongs to
$\cat{A}^{<s}$ by
Lemma~\ref{lem:left-localization-is-closed-under-limits}, so
\[
\lim_r(\lim_iX_i^{<s})^{<r}=\lim_iX_i^{<s}.
\]
Now $\lim_s\lim_iX_i^{<s}=\lim_i\complete{X}_i$, so we have proved
that $\operatorname*{\complete\lim}_iX_i=\lim_i\complete{X}_i$ as
desired.
\end{proof}

We have proved the following.
\begin{proposition}\label{prop:complete-with-bounded-sequential-limits}
Let $\cat{A}$ be a filtered stable category with uniformly bounded
sequential limits.
Then $\complete{\cat{A}}$ is the completion of $\cat{A}$, with
complementary right localization $\cat{A}_{\ge\infty}$ as a full
subcategory of $\cat{A}$.
\end{proposition}

\begin{corollary}
If sequential limits are uniformly bounded in $\cat{A}$, then
$\cat{A}$ is complete if and only if $\cat{A}_{\ge\infty}\equivwith\zero$.
\end{corollary}

\subsection{The completion as a complete category}
\setcounter{subsubsection}{-1}

When $\complete{\cat{A}}$ is the completion of a filtered stable
category $\cat{A}$, then it will be useful if the completion is itself
a complete filtered stable category.
We would like to first consider a sufficient condition for the
completion to be a \emph{stable} category.
We have found a sufficient condition for a general localization in
Proposition~\ref{prop:exact-localization-is-stable}.

\begin{definition}
Let $\cat{A}$ be a filtered stable category with $\complete{\cat{A}}$
its completion.
Then we say that the completion is \kore{exact} if
$\complete{\cat{A}}$ is an exact left localization of $\cat{A}$.
\end{definition}

We shall look for a sufficient condition for the completion to be
exact.

\begin{definition}
Let $\cat{A}$ be a filtered stable category.
An integer $\omega$
is said to be a \kore{uniform lower bound for loops} in $\cat{A}$ if
for every integer $r$, and for every object of the full subcategory
$\cat{A}_{\ge r}$ of $\cat{A}$, its loop in $\cat{A}$ belongs to
$\cat{A}_{\ge r+\omega}$.
We say that \kore{loops are uniformly bounded} in
$\cat{A}$ if loops in $\cat{A}$ have a uniform lower bound.
\end{definition}

\begin{remark}\label{rem:on-bounded-loops}
Loops are uniformly bounded if the action of the
category of (finite) spectra on $\cat{A}$ by tensoring, is compatible
with the filtrations (on the category of spectra and on $\cat{A}$) in
a way similar to (or slightly more general than) the way we consider
in Definition \ref{def:compatible-pairing}.
In this case, the suspension functor raises the filtration, as we
shall consider in Definition \ref{def:translated-filtration}.
\end{remark}

\begin{example}\label{ex:bounded-loop}
$\omega$ can be taken as $-1$ if the filtration is a t-structure on
$\cat{A}$.

$\omega$ can be taken as $0$ for Goodwillie's filtration.
In fact, all localizations are exact in this filtration.
\end{example}

\begin{remark}\label{rem:lower-bound-is-negative}
An integer $\omega\gneq 0$ cannot be a uniform lower bound for loops
unless $\cat{A}_{\ge r}$ for all $r$ are the same subcategory of
$\cat{A}$.
Indeed, $\Loop^{-1}=\susp$ maps $\cat{A}_{\ge r}$ into $\cat{A}_{\ge
  r}$ by Lemma~\ref{lem:left-localization-is-closed-under-limits}.
\end{remark}

\begin{lemma}\label{lem:finite-limit-bounded}
Let $\cat{A}$ be a filtered stable category.
If loops are uniformly bounded in $\cat{A}$, then for any finite
category $K$, limits of $K$-shaped diagrams are uniformly bounded in
$\cat{A}$.
\end{lemma}
\begin{proof}
By Corollary \ref{cor:closure-under-extension}, $\omega$ is a uniform
lower
bound for loops if and only if it is a uniform lower bound for
\emph{fibres} in the similar sense.
Indeed, we may assume $\omega\ge 0$ by
Remark~\ref{rem:lower-bound-is-negative}, and if $W\to X\to Y$ is a
fibre sequence in $\cat{A}$, then there is a fibre sequence $\Loop
Y\to W\to X$.

It follows again from Corollary \ref{cor:closure-under-extension},
that the uniform lower bound of fibres more generally bounds fibre
products.

The result now follows from
Corollary~\ref{cor:left-localization-is-closed-under-sum} (applied in
the opposite category) and the arguments of the proof of
Corollary~4.4.2.4 of \cite{topos}.
\end{proof}

\begin{lemma}\label{lem:exact-completion-from-bounded-loop}
Let $\cat{A}$ be a filtered stable category with $\complete{\cat{A}}$
its completion.
If loops are uniformly bounded in $\cat{A}$, then the completion is
exact.
\end{lemma}
\begin{proof}
By Propositions~\ref{prop:complement-of-completion} and
\ref{prop:limit-in-localization}, it suffices to prove that the full
subcategory $\cat{A}_{\ge\infty}$ of $\cat{A}$ is closed
under finite limits in $\cat{A}$.

Let $K$ be a finite category, and let $X$ be a $K$-shaped diagram in
the full subcategory $\cat{A}_{\ge\infty}$ of $\cat{A}$.
Then we would like to prove that $\lim_KX$ belongs to
$\cat{A}_{\ge\infty}$.
However, for every $r$, $\lim_KX$ does belongs to $\cat{A}_{\ge r}$
since $X$ is in particular a diagram in $\cat{A}_{\ge r-k}$ for a
uniform bound $k$ of $K$-shaped limits, which exists by
Lemma~\ref{lem:finite-limit-bounded}.
\end{proof}

\begin{definition}
Let $\cat{A}$, $\cat{B}$ be filtered stable categories, and let
$F\colon\cat{A}\to\cat{B}$ be an exact functor.
Then we say that an integer $b$ is a \kore{lower bound} of
$F$ if for every $r$, $F$ takes the full subcategory $\cat{A}_{\ge r}$
of the source to the full subcategory $\cat{B}_{\ge r+b}$ of the
target.

We say that $F$ is \kore{bounded below} if it has a lower bound.

\kore{Upper} bound/boundedness of $F$ is defined as the lower
bound/boundedness of $F\colon\cat{A}^\op\to\cat{A}^\op$ with respect
to the dual filtration on $\cat{A}^\op$
(Remark~\ref{rem:dual-filtration}).
\end{definition}

Thus loops are uniformly bounded in $\cat{A}$ if the functor
$\Loop\colon\cat{A}\to\cat{A}$ is bounded below.
$\Loop$ also has $0$ as an upper bound by
Lemma~\ref{lem:left-localization-is-closed-under-limits}.

We obtain from the following, that a uniform lower bound for loops
also gives an upper bound of the suspension functor.

\begin{lemma}\label{lem:adjoint-preserving-localization}
Let $\cat{A}$, $\cat{B}$ be filtered stable categories, and let
$F\colon\cat{A}\to\cat{B}$ be a functor which has a right adjoint
$G$.
Then an integer $b$ is a lower bound of $F$ if and only if $-b$ is an
upper bound of $G$.
\end{lemma}
\begin{proof}
For an integer $b$, the composite
\[\begin{tikzcd}[column sep=large]
    \cat{A}_{\ge r}\arrow[r, hook]&\cat{A}\arrow[r, "F"]
    &\cat{B}\arrow[r, "\blank^{<r+b}"]&\cat{B}^{<r+b}
  \end{tikzcd}\]
is null if and only if the composite of the right adjoints
\[\begin{tikzcd}[column sep=large]
    \cat{A}_{\ge r}&\cat{A}\arrow[l, "\blank_{\ge r}"']
    &\cat{B}\arrow[l, "G"']&\cat{B}^{<r+b}\arrow[l, hook']
  \end{tikzcd}\]
is null, since either adjoint of a null functor is null.
\end{proof}

We obtain the following as a by-product.

\begin{proof}[Alternative proof of
  Lemma~\ref{lem:exact-completion-from-bounded-loop}]
By Proposition~\ref{prop:limit-in-localization} and (the dual case
of) Proposition~\ref{prop:exact-localization-is-stable}, it suffices
to prove that the full subcategory $\complete{\cat{A}}$ of $\cat{A}$
is closed under finite colimits in $\cat{A}$.

Let $K$ be a finite category, and let $X$ be a diagram in the full
subcategory $\complete{\cat{A}}$ of $\cat{A}$.
Then we would like to prove that $\colim_KX$ is complete.

However, $\colim_KX=\lim_r\colim_KX^{<r}$, and $\colim_KX^{<r}$
belongs to $\cat{A}^{<r+k}$ for a uniform bound $k$ of $K$-shaped
colimits, which exists by
Lemma~\ref{lem:adjoint-preserving-localization} and
Lemma~\ref{lem:finite-limit-bounded}, applied in the opposite
category.
The result follows from Corollary \ref{cor:complete-limit}.
\end{proof}

\begin{proposition}\label{prop:filtration-of-completion}
Let $\cat{A}$ be a filtered stable category with $\complete{\cat{A}}$
its completion.
If the completion is exact, then the canonical tower
\[
\complete{\cat{A}}\longto\cdots\longto\cat{A}^{<r}\longto\cat{A}^{<r-1}\longto\cdots
\]
makes $\complete{\cat{A}}$ into a complete filtered stable category.
\end{proposition}
\begin{proof}
As we have remarked in Example \ref{ex:bouded-above-are-complete}, for
every $r$, $\cat{A}^{<r}\sub\complete{\cat{A}}$ as full subcategories
of $\cat{A}$.
It follows that the restriction to $\complete{\cat{A}}$ of the
localization functor $\cat{A}\to\cat{A}^{<r}$ is a left localization.
A complementary right localization to this is given by
$\cat{A}_{\ge r}\intersect\complete{\cat{A}}$.

In order to verify that $\complete{\cat{A}}$ is complete with respect
to this filtration, let $X$ be an object of $\complete{\cat{A}}$.
Then in $\cat{A}$, we have that the canonical map
\begin{equation}\label{eq:completion-map-of-complete-object}
X\longto\lim_rX^{<r}
\end{equation}
is an equivalence, where the limit is taken in $\cat{A}$.
It follows that this limit, since it consequently belongs to
$\complete{\cat{A}}$, is also a limit in $\complete{\cat{A}}$ of the
same sequence.
(Alternatively, one could apply
Proposition~\ref{prop:complement-of-completion} and
Lemma~\ref{lem:left-localization-is-closed-under-limits}.)
Therefore, the equivalence
(\ref{eq:completion-map-of-complete-object}) shows that $X$ is
complete with respect to our filtration of $\complete{\cat{A}}$, and
this verifies the completeness of $\complete{\cat{A}}$
(Definition~\ref{def:completion}).
\end{proof}

\begin{lemma}\label{lem:bound-of-limits-in-completion}
Let $\cat{A}$ be a filtered stable category with $\complete{\cat{A}}$
its exact completion.
Then any class of limits which exist in $\cat{A}$ (and therefore also
in $\complete{\cat{A}}$ by
Lemma \ref{lem:left-localization-is-closed-under-limits}) and are
uniformly bounded, have the same uniform lower bound in
$\complete{\cat{A}}$.
\end{lemma}
\begin{proof}
Lemma \ref{lem:left-localization-is-closed-under-limits} in fact
states that $\complete{\cat{A}}$ is closed under the limits which
exists in $\cat{A}$.
The result follows since the full subcategory $\complete{\cat{A}}_{\ge
  r}$ in the filtration of $\complete{\cat{A}}$ is just $\cat{A}_{\ge
  r}\intersect\complete{\cat{A}}$ as a full subcategory of $\cat{A}$.
\end{proof}

\subsection{Totalization}\label{sec:totalization}
\setcounter{subsubsection}{-1}

\subsubsection{}
 
In this section, we shall prove a technical result which will be very
useful for our study of the Koszul duality.

Let $\Simpface$ denote the subcategory of the
category $\Simp$ of combinatorial simplices, where only face maps
(maps \emph{strictly} preserving the order of vertices) are included.
A covariant functor $X^\dot\colon\Simpface\to\cat{A}$ is a
cosimplicial object `without degeneracies' of $\cat{A}$.
Its \emph{totalization} $\Tot X^\dot$ is by definition, the limit over
$\Simpface$ of the diagram $X^\dot$.

\begin{proposition}\label{prop:tatalization-preserved}
Let $\cat{A}$, $\cat{B}$ be filtered stable categories which have
sequential limits, and let
$F\colon\cat{A}\to\cat{B}$ be an exact functor which is bounded below.
Assume that loops and sequential limits are uniformly bounded in
$\cat{A}$, and $\complete{\cat{B}}$ is the completion of $\cat{B}$.

Let $X^\dot\colon\Simpface\to\cat{A}$ be such that there exists a
sequence $r=(r_n)_n$ of integers, tending to $\infty$ as $n\to\infty$,
such that for a uniform lower bound $\omega$ for loops, and for every
$n$, $X^n$ belongs to $\cat{A}_{\ge-\omega n+r_n}$.
Then the canonical map
\[
F(\Tot X^\dot)\longto\Tot FX^\dot
\]
is an equivalence after completion.
\end{proposition}
\begin{proof}
According to the sequence of full subcategories
\[
\Simpface\:\includes\:\cdots\:\includes\:\Simpface^{\le
  n}\:\includes\:\Simpface^{\le n-1}\:\includes\:\cdots,
\]
where objects of $\Simpface^{\le n}$ are simplices of dimension at
most $n$, we have the sequence
\[
\Tot X^\dot\longto\cdots\longto\sk_n\Tot X^\dot\longto\sk_{n-1}\Tot
X^\dot\longto\cdots
\]
such that $\Tot X^\dot=\lim_n\sk_n\Tot X^\dot$, where ``$\sk_n\Tot$''
is a single symbol representing the operation of taking the limit over
$\Simpface^{\le n}$.

It is standard that the fibre of the map $\sk_n\Tot
X^\dot\to\sk_{n-1}\Tot X^\dot$ is equivalent to $\Loop^nX^n$.
It follows from our assumption that this belongs to $\cat{A}_{\ge
  r_n}$.
It follows that the fibre of the map $\Tot X^\dot\to\sk_n\Tot X^\dot$
belongs to $\cat{A}_{\ge r_n+d}$ for $d$ a uniform lower bound for
sequential limits.

It follows that the fibre of the map $F(\Tot X^\dot)\to\sk_n\Tot
FX^\dot$ belongs to $\cat{B}_{\ge r_n+d+b}$ for a bound $b$ of
$F$.
By taking the limit over $n$, we obtain the result from
Corollary \ref{cor:equivalence-of-limits-after-completion}.
\end{proof}

\subsubsection{}

Similarly, we can consider simplicial objects `without degeneracies'
and their geometric realizations.

\begin{proposition}\label{prop:realization-preserved}
Let $\cat{A}$, $\cat{B}$ be filtered stable categories, and let
$F\colon\cat{A}\to\cat{B}$ be an exact functor which is bounded below.
Assume that $\complete{\cat{B}}$ is the completion of $\cat{B}$.
Let $X_\dot\colon\Simpface^\op\to\cat{A}$ be such that there exists a
sequence $r=(r_n)_n$ of integers, tending to $\infty$ as $n\to\infty$,
such that for every $n$, $X^n$ belongs to $\cat{A}_{\ge r_n}$.
Then the canonical map
\[
|FX_\dot|\longto F|X_\dot|
\]
is an equivalence after completion.
\end{proposition}
\begin{proof}
The proof of this is simpler.
One simply notes that the full subcategory $\cat{A}_{\ge r}$ of
$\cat{A}$ is closed under any colimit by
Lemma~\ref{lem:left-localization-is-closed-under-limits}, and
similarly in $\cat{B}$.
It follows that the fibre of the map in question belongs to
$\cat{B}_{\ge \infty}$, and we conclude by applying
Proposition~\ref{prop:complement-of-completion} and
Lemma~\ref{lem:local-equivalence}.
\end{proof}

\section{Monoidal filtered stable category}
\label{sec:monoidal-filtered}
\setcounter{subsection}{-1}
\setcounter{equation}{-1}
\subsection{Pairing in filtered stable categories}
\setcounter{subsubsection}{-1}

\begin{definition}
Let $\cat{A}$, $\cat{B}$, $\cat{C}$ be stable categories.
Then, a \kore{pairing} from $\cat{A}$, $\cat{B}$ to $\cat{C}$ is a
functor
\[
\pairing{\;{,}\;}\colon\cat{A}\times\cat{B}\longto\cat{C}
\]
which is exact in each variable.
\end{definition}

\begin{definition}\label{def:compatible-pairing}
Let $\cat{A}$, $\cat{B}$, $\cat{C}$ be filtered stable categories.
Then, a pairing
$\pairing{\;{,}\;}\colon\cat{A}\times\cat{B}\to\cat{C}$ is
\kore{compatible} with the filtrations if for any $r$, $s$, it takes
$\cat{A}_{\ge r}\times\cat{B}_{\ge s}$ to $\cat{C}_{\ge r+s}$.
\end{definition}

\begin{example}
Let $\cat{S}$ denote the stable category of finite spectra, filtered
by connectivity.
Let $\cat{A}$ be another stable category with a t-structure.
Then the pairing $\cat{S}\times\cat{A}\to\cat{A}$ given by tensoring
is compatible with the filtrations.
\end{example}

\begin{remark}\label{rem:pairing-bounded-below}
More generally, we may say that the pairing is \kore{bounded below} if
there is a finite integer $d$ such that the pairing takes
$\cat{A}_{\ge r}\times\cat{B}_{\ge s}$ to $\cat{C}_{\ge r+s+d}$.
In this case, if $d$ can be taken only as a negative number, then the
pairing is not compatible with the filtrations in the above sense.
However, this can be corrected by reindexing the filtrations in any
suitable way.
\end{remark}

It might seem natural to further require the pairing to preserve
sequential limits (variable-wise).
However, this condition is too strong to require in practice.
We shall see that the compatibility defined above ensures that certain
sequential limits are preserved (up to completion).
This will turn out to be useful for applications.

\begin{definition}
Let $\cat{A}$ be a filtered stable category.
Then an object $X$ of $\cat{A}$ is said to be \kore{bounded below} in
the filtration if there exists an integer $r$ such that
$X\in\cat{A}_{\ge r}$ in $\cat{A}$.
\end{definition}

Let $\cat{A}$, $\cat{B}$, $\cat{C}$ be filtered stable categories, and
consider a pairing
$\pairing{\;{,}\;}\colon\cat{A}\times\cat{B}\to\cat{C}$,
compatible with the filtrations.
The following is an immediate consequence of
Lemma \ref{lem:vanishing-completion}.
\begin{lemma}\label{lem:vanishing-pairing}
Assume that $\complete{\cat{C}}$ is the completion of $\cat{C}$.

Suppose given an inverse system
\[
\cdots\longfrom X_i\longfrom X_{i+1}\longfrom\cdots
\]
in $\cat{A}$, and suppose there is a sequence
$(r_i)_i$ of integers, tending to $\infty$ as $i\to\infty$,
such that for every $i$, $X_i$ belongs to $\cat{A}_{\ge r_i}$.
Then for every \textbf{bounded below} $Y\in\cat{B}$,
$\lim_i\pairing{X_i,Y}$ belongs to $\cat{C}_{\ge\infty}$.
\end{lemma}

\begin{proposition}\label{prop:continuity-of-pairing}
Let $\pairing{\;{,}\;}\colon\cat{A}\times\cat{B}\to\cat{C}$ be a
pairing on filtered stable categories, compatible with the
filtrations.
Assume that $\complete{\cat{C}}$ is the completion of $\cat{C}$.
Assume also
that sequential limits exist and are uniformly bounded in $\cat{A}$.

Suppose given an inverse system
\[
\cdots\longfrom X_i\longfrom X_{i+1}\longfrom\cdots
\]
in $\cat{A}$, and suppose there is a sequence
$(r_i)_i$ of integers, tending to $\infty$ as $i\to\infty$, such that
for every $i$, the fibre of the map $X_{i+1}\to X_i$ belongs to
$\cat{A}_{\ge r_i}$.
Then for every $Y\in\cat{B}$, if $Y$ is bounded below, then the
induced map
\[
\pairing{\lim_iX_i,Y}\longto\lim_i\pairing{X_i,Y}
\]
is an equivalence after completion.
\end{proposition}
\begin{proof}
This follows from
Corollary \ref{cor:equivalence-of-limits-after-completion}.
\end{proof}

\subsection{Monoidal structure on a filtered category}
\label{sec:monoidal-filtration}
\setcounter{subsubsection}{-1}

\subsubsection{}
 
By a \kore{monoidal structure} on a stable category $\cat{A}$, we mean
a monoidal structure on the underlying category of $\cat{A}$ whose
multiplication operations are exact in each variable.

\begin{definition}\label{def:filtered-tensor}
Let $\cat{A}$ be a filtered stable category, and let $\tensor$ be a
monoidal structure on the stable category (underlying) $\cat{A}$.
We say that the monoidal structure is \kore{compatible} with the
filtration on $\cat{A}$ if for every finite totally ordered set $I$,
and every sequence $r=(r_i)_{i\in I}$ of integers, the functor
$\Tensor_I\colon\cat{A}^I\to\cat{A}$ takes the full subcategory
$\prod_{i\in I}\cat{A}_{\ge r_i}$ of the source, to
the full subcategory $\cat{A}_{\ge\sum_Ir}$ of the target.

We call a filtered stable category $\cat{A}$ equipped with a
compatible monoidal structure a \kore{monoidal} filtered stable
category.
If the monoidal structure is symmetric, then it will just be a
\kore{symmetric} monoidal filtered category.
\end{definition}

In other words, a symmetric monoidal filtered stable category is just
a commutative monoid object of a suitable multicategory of filtered
stable categories.

\begin{example}
In the case where $\cat{A}$ is a functor category with Goodwillie's
filtration, if the target category of the functors is a symmetric
monoidal stable
category, then the pointwise symmetric monoidal structure \emph{on
  $\cat{A}^\op$} is compatible with the filtration.
Note Remark~\ref{rem:dual-filtration}.
\end{example}

\begin{remark}\label{rem:monoidal-structure-bounded-below}
More generally, we may say that the monoidal structure is
\kore{bounded below} if the unit $\unity$ and the pairing
$\cat{A}^2\xrightarrow{\tensor}\cat{A}$ are bounded below.
All the results we consider in the following on monoidal filtered
categories will be valid for filtered stable categories with a bounded
below monoidal structure, after making suitable (and straightforward)
modifications.

We shall only state the results for monoidal filtered categories in
our sense, in order to keep the exposition simple.
\end{remark}

\begin{remark}\label{rem:break-of-duality}
Even though both filtration and monoidal structure are self-dual
notion on a stable category, the boundedness below of the monoidal
structure is not self-dual.
Namely, boundedness below in $\cat{A}^\op$ means boundedness
\emph{above} in $\cat{A}$.
Instead, Lemma~\ref{lem:adjoint-preserving-localization} implies that
the internal hom functor would have suitable boundedness above on a
symmetric monoidal stable category.
\end{remark}

\begin{corollary}\label{cor:continuity-of-monoidal-product}
Let $\cat{A}$ be a monoidal filtered stable category with uniformly
bounded sequential limits.

Suppose given an inverse system
\[
\cdots\longfrom X_i\longfrom X_{i+1}\longfrom\cdots
\]
in $\cat{A}$, and suppose there is a sequence
$(r_i)_i$ of integers, tending to $\infty$ as $i\to\infty$, such that
for every $i$, the fibre of the map $X_{i+1}\to X_i$ belongs to
$\cat{A}_{\ge r_i}$.
Then for every $Y\in\cat{A}$, if $Y$ is bounded below, then the
induced map
\[
(\lim_iX_i)\tensor Y\longto\lim_i(X_i\tensor Y)
\]
is an equivalence after completion.
\end{corollary}

\subsubsection{}

\begin{definition}\label{def:completable}
Let $\cat{A}$ be a monoidal filtered stable category with
$\complete{\cat{A}}$ completing the filtration.
Then we say that the monoidal structure is \kore{completable} if there
is a monoidal structure on $\complete{\cat{A}}$ such that the
completion functor $\cat{A}\to\complete{\cat{A}}$ is monoidal.
\end{definition}
\begin{remark}\label{rem:unique-completion-of-tensor}
Together with a monoidal structure of the completion functor, the
monoidal structure on $\complete{\cat{A}}$ will be uniquely
determined.
\end{remark}

\begin{lemma}\label{lem:completable-monoidal-structure}
Let $\cat{A}$ be a filtered stable category with $\complete{\cat{A}}$
being its exact completion.
Then a monoidal structure of $\cat{A}$ is completable if and only if
for every integer $n\ge 0$, the monoidal product $\Tensor_{i=0}^nX_i$
for a sequence $X_i$, $0\le i\le n$, of objects of $\cat{A}$
necessarily belongs to the full subcategory $\cat{A}_{\ge\infty}$ of
$\cat{A}$ whenever $X_i\in\cat{A}_{\ge\infty}$ for some $i$.
\end{lemma}
\begin{proof}
Necessity is obvious. 

For sufficiency, consider for every integer $n\ge 0$, the objects of
$\cat{A}^n$ which are required in the stated condition, to be sent
by the monoidal product functor
$\tensor\colon\cat{A}^n\to\cat{A}$, to the full subcategory
$\cat{A}_{\ge\infty}$.
For example, there is no such object in the case $n=0$.
Denote by $\cat{J}_n$ the full subcategory of $\cat{A}^n$ consisting
of all these objects.
Namely, a sequence of objects of $\cat{A}$ belongs to $\cat{J}_n$ if
and only if at least one of the entries of the sequence comes from
$\cat{A}_{\ge\infty}$.

Then we have for every $n$, that the postcomposition functor
$\Fun^\exact(\cat{A}^n,\cat{A})\to\Fun^\exact(\cat{A}^n,\complete{\cat{A}})$
with the localization functor (where $\Fun^\exact$ indicates the
functors which are exact in each variable separately) sends the full
subcategory of the source consisting of those functors which send
$\cat{J}_n\sub\cat{A}^n$ to the full subcategory $\cat{A}_{\ge\infty}$
of $\cat{A}$, over to the full subcategory
$\Fun^\exact(\complete{\cat{A}}^n,\complete{\cat{A}})$ of the target
$\Fun^\exact(\cat{A}^n,\complete{\cat{A}})$, embedded by the functor
of precomposition again with the localization functors
$\cat{A}\to\complete{\cat{A}}$.

Sufficiency of the condition is an immediate consequence.
\end{proof}

From the proof, the monoidal operations on $\complete{\cat{A}}$ in the
completable case can be written as the composites
\[\begin{tikzcd}
\complete{\cat{A}}^n\arrow[hook]{r}&\cat{A}^n\arrow{r}{\tensor}&\cat{A}\arrow{r}{\complete{\blank}}&\complete{\cat{A}}.
\end{tikzcd}\]

\begin{proposition}\label{prop:completed-monoidal-filtration}
Let $\cat{A}$ be a monoidal filtered stable category with
$\complete{\cat{A}}$ its exact completion.
If the monoidal structure on $\cat{A}$ is completable, then
$\complete{\cat{A}}$ with the induced structures is a monoidal
(complete) filtered stable category.
\end{proposition}
\begin{proof}
The variable-wise exactness of the induced monoidal structure on
$\complete{\cat{A}}$ follows from the description of the monoidal
product functor after the proof of
Lemma \ref{lem:completable-monoidal-structure} (see
Remark \ref{rem:unique-completion-of-tensor}).

We further need to prove that this monoidal structure is compatible
with the induced filtration on $\complete{\cat{A}}$ (see
Proposition \ref{prop:filtration-of-completion}).
This follows since $\complete{\cat{A}}_{\ge r}\sub\cat{A}_{\ge r}$ as
full subcategories of $\cat{A}$, and the completion functor
$\cat{A}\to\complete{\cat{A}}$ takes the full subcategory
$\cat{A}_{\ge r}$ of the source, to the full subcategory
$\complete{\cat{A}}_{\ge r}$ of the target.
\end{proof}

\begin{lemma}
Let $\cat{A}$ be as in
Proposition \ref{prop:completed-monoidal-filtration}.
If the monoidal multiplication functor on $\cat{A}$ preserves variable-wise,
a certain class of colimits (specified as in
Proposition \ref{prop:limit-in-localization}, and assumed to exist),
then so does the completed monoidal operation on $\complete{\cat{A}}$
\textbf{if} for every $r$, the full subcategory $\cat{A}^{<r}$ of
$\cat{A}$ are closed under the class of colimits taken in $\cat{A}$.
\end{lemma}
\begin{proof}
In view of the description of the monoidal multiplication functor
after the
proof of Lemma \ref{lem:completable-monoidal-structure}, it suffices
to prove under our assumption, that the inclusion functor
$\complete{\cat{A}}$ preserves the class colimits in question.
By Proposition \ref{prop:limit-in-localization}, this condition is
equivalent to that the localization functor
$\blank_{\ge\infty}\colon\cat{A}\to\cat{A}_{\ge\infty}$ preserves the
class of colimits in question.

Recall that $\cat{A}_{\ge\infty}=\lim_r\cat{A}_{\ge r}$.
Since this limit is along colimit preserving functors, colimits in
$\lim_r\cat{A}_{\ge r}$ is object-wise.
Therefore, it suffices to show for every $r$, that $\blank_{\ge
  r}\colon\cat{A}\to\cat{A}_{\ge r}$ preserves colimits.

We conclude by invoking Proposition \ref{prop:limit-in-localization}
again.
\end{proof}

\subsection{The filtered category of filtered objects}
\label{sec:examples-of-filtered-categories}
\setcounter{subsubsection}{-1}

\subsubsection{}
 
In this section, we shall give a simple example of a filtered stable
category, for which limits of any kind are uniformly bounded by $0$.
We also show how this filtered stable category may have a completable
symmetric monoidal structure.

\subsubsection{}

Let us denote by $\Stable$ the following symmetric
($2$-)multicategory.
Its object is a stable category.
Given a family $\cat{A}=(\cat{A}_s)_{s\in S}$ of stable categories
indexed by a finite set $S$, and a stable category $\cat{B}$, we
define a multimap $\cat{A}\to\cat{B}$ to be a functor
$\prod_S\cat{A}\to\cat{B}$ which is exact in each variable.
(Note that the condition is vacuous when there is no variable, i.e.,
when $S$ is empty and the product is one point.)

Let $\Z$ be the category
\[
\cdots\longfrom n\longfrom n+1\longfrom\cdots
\]
defined by the poset of integers.
This is a commutative monoid (in the category of poset), so the
functor category $\Fun(\Z,\Stable)$ is a symmetric multicategory.

Let $\cat{B}$ be an object of $\Fun(\Z,\Stable)$, and let $\cat{A}$
be the category of lax morphisms $*\to\cat{B}$ in $\Fun(\Z,\Cat)$.
To be more precise about the variance in the definition of a lax
functor here, we consider the functor
$y\colon\Z\to\Cat$, $n\mapsto\Z_{/n}$, and define $\cat{A}$
to be the category of (genuine, rather than lax) morphisms
$y\to\cat{B}$.

\begin{definition}
In the case where the sequence $\cat{B}$ is constant at a stable
category $\cat{C}$, we call an object of $\cat{A}$ a \kore{filtered
  object} of $\cat{C}$, so $\cat{A}$ will be the \emph{category of
  filtered objects} of $\cat{C}$.
\end{definition}

$\cat{A}$ filtered as follows.

$\cat{B}$ is a sequence of stable categories
\[
\cdots\overset{L}{\longfrom}\cat{B}_n\overset{L}{\longfrom}\cat{B}_{n+1}\overset{L}{\longfrom}\cdots,
\]
and $\cat{A}$ is the category of sequences which we shall typically
express as
\[
\cdots\longfrom F^nX\longfrom F^{n+1}X\longfrom\cdots,
\]
where $F^nX\in\cat{B}_n$, and the arrow $F^nX\from F^{n+1}X$ is meant
to be a map $F^nX\from LF^{n+1}X$ in $\cat{B}_n$.

We let $\cat{A}_{\ge r}$ be the category of sequences
\[
F^rX\longfrom F^{r+1}X\longfrom\cdots,
\]
and $\blank_{\ge r}\colon\cat{A}\to\cat{A}_{\ge r}$ to be the functor
which forgets objects $F^nX$ for $n<r$.
$\blank_{\ge r}$ is a right localization which has a complementary
left localization, which we denote by
$\blank^{<r}\colon\cat{A}\to\cat{A}^{<r}$.

\begin{lemma}\label{lem:zero-bounds-all-limits}
Let $I$ be a small category.
If for every $n$, the category $\cat{B}_n$ has the limit of every
$I$-shaped diagram, then $\cat{A}$ has all $I$-shaped limits, and
these limits are uniformly bounded by $0$.
\end{lemma}
\begin{proof}
This is obvious, since $I$-shaped limits in $\cat{A}$ will be given
object-wise.
\end{proof}

\begin{remark}\label{rem:stable-exact-completion-of-filtered-objects}
Note that this for finite limits implies that completion of $\cat{A}$
would be exact by Lemma~\ref{lem:exact-completion-from-bounded-loop}.
$\cat{A}$ will have all sequential limits (which are given
object-wise) if each $\cat{B}_n$ has all sequential limits.
\end{remark}

\subsubsection{}

Now suppose $\cat{B}\in\Fun(\Z,\Stable)$ is a symmetric monoid
object.
Concretely, this means there are given functors
$\tensor\colon\cat{B}_i\times\cat{B}_j\to\cat{B}_{i+j}$ (in a way
symmetric in $i$ and $j$) etc.~which is exact in each variable, and a
unit object $\unity\in\cat{B}_0$, with data of compatibility with
$L$'s in the sequence $\cat{B}$.

Moreover, assume the following.
\begin{itemize}
\item For every $n$, $\cat{B}_n$ has all small colimits.
\item For every $n$, $L\colon\cat{B}_n\to\cat{B}_{n-1}$ preserves
  colimits.
\item The monoidal multiplication functors preserve colimits variable-wise.
\end{itemize}

In this case, $\cat{A}$ inherits a symmetric monoidal structure.
Namely, if $X=(F^nX)_n$ and $Y=(F^nY)_n$ are objects of $\cat{A}$,
then we have $X\tensor Y$ defined by
\[
F^n(X\tensor Y)=\colim_{i+j\ge n}L^{i+j-n}(F^iX\tensor F^jY),
\]
the colimit taken in $\cat{B}_n$.
This monoidal multiplication preserves colimits variablewise.

\begin{proposition}\label{prop:filtered-objects-monoidal}
The symmetric monoidal structure on $\cat{A}$ is
compatible with the filtration on $\cat{A}$.
\end{proposition}
\begin{proof}
Let $r$, $s$ be integers, and let $X\in\cat{A}_{\ge r}$ and
$Y\in\cat{A}_{\ge s}$.
Then we would then like to prove $X\tensor Y\in\cat{A}_{\ge r+s}$.

In terms of the sequence defining $X$ and $Y$, the given conditions
are that the map $F^iX\from L^{r-i}F^rX$ is an equivalence for $i\le
r$, and similarly for $Y$.
Under these assumptions, we need to prove that the map $F^n(X\tensor
Y)\from L^{r+s-n}F^{r+s}(X\tensor Y)$ is an equivalence for $n\le r+s$.

By definition, $F^n(X\tensor Y)$ was the colimit over $i,\,j$ such
that $i+j\ge n$, of $L^{i+j-n}(F^iX\tensor F^jY)$.
It suffices to prove that, for $n\le r+s$, this colimit is the same as
the colimit of $L^{i+j-n}(F^iX\tensor F^jY)$ over $i,\,j$ such that
$i\ge r$ and $j\ge s$.
However, the given assumptions imply that the diagram over $i,\,j$
such that $i+j\ge n$, is the left Kan extension of its restriction to
$i,\,j$ such that $i\ge r$ and $j\ge s$, since the assumptions imply
that the map $F^iX\tensor F^jY\from F^{\max\{i,r\}}X\tensor
F^{\max\{j,s\}}Y$ (in $\cat{B}_{i+j}$; we have omitted $L$ from the
notation) will be an equivalence for all $i,\,j$.
The result follows.
\end{proof}

\begin{proposition}\label{prop:complete-objects-monoidal}
If each $\cat{B}_n$ is closed under the sequential limit, so
$\complete{\cat{A}}$ is the exact completion of $\cat{A}$, then the
monoidal structure on $\cat{A}$ is completable.
\end{proposition}
\begin{proof}
We want to show that if $X\in\cat{A}_{\ge\infty}$ and $Y\in\cat{A}$,
then $X\tensor Y\in\cat{A}_{\ge\infty}$.

The given condition is the same as that the map $F^nX\from F^{n+1}X$
is an equivalence for every $n$.
We then want to prove that the map
\[
F^n(X\tensor Y)=\colim_{i+j\ge n}F^iX\tensor
F^jY\longfrom\colim_{i+j\ge n+1}F^iX\tensor F^jY=F^{n+1}(X\tensor Y)
\]
is an equivalence for every $n$.

However, an inverse to this map can be constructed as the colimit
\[
\colim_{i+j\ge n}F^iX\tensor F^jY\longto\colim_{i+j\ge n}F^{i+1}X\tensor F^jY
\]
of the maps induced from the inverses $F^iX\to F^{i+1}X$ to the given
equivalences.
\end{proof}

\subsection{Modules over an algebra in filtered stable category}
\setcounter{subsubsection}{-1}

\subsubsection{}
 
Further examples of filtered stable categories will be found by
considering modules over an algebra in a filtered stable category.
We shall investigate them further.

\subsubsection{}

Let us start from the situation of general localization of a stable
category.

Thus, let $\cat{A}$ be a stable category, and let a monoidal
structure $\tensor$ on $\cat{A}$ be given.
Recall that we assume by convention that the monoidal multiplication
is exact in each variable.

Let us further assume given a left localization
$\blank_\leftlocal\colon\cat{A}\to\cat{A}_\leftlocal$ of $\cat{A}$
with a complementary right localization
$\blank_\rightlocal\colon\cat{A}\to\cat{A}_\rightlocal$.

We assume given an associative algebra $A$ in $\cat{A}$, and
would like to have a corresponding localization of the category
$\Mod_A$ of (say, right) $A$-modules, in a natural way.
A sufficient condition so one can do this is that the functor
$-\tensor A\colon\cat{A}\to\cat{A}$ take $\cat{A}_\rightlocal$ to
$\cat{A}_\rightlocal$.
(There is no difference if $\cat{A}$ is not assumed to be monoidal,
but it is given an action by any monad, in place of an action of an
algebra object in $\cat{A}$.
For our applications, we do not need to use this language.)

Indeed, if $A$ satisfies this condition, then for any object $X$ of
$\cat{A}_\rightlocal$ and $Y$ of $\cat{A}$, and for any integer $n\ge
0$, we have that the induced map
\[
\Map(X\tensor A^{\tensor n},Y_\rightlocal)\longto\Map(X\tensor
A^{\tensor n},Y)
\]
is an equivalence.
It follows that the category
$\Mod_{A,\rightlocal}:=\Mod_A(\cat{A}_\rightlocal)$ of $A$-modules in
$\cat{A}_\rightlocal$, is a full subcategory of
$\Mod_A(\cat{A})=\Mod_A$ by the functor induced from
$\cat{A}_\rightlocal\into\cat{A}$, and is a right localization of
$\Mod_A$.
It further follows that the square
\[\begin{tikzcd}
\Mod_{A,\rightlocal}\arrow[hook]{r}\arrow{d}
&\Mod_A\arrow{d}\\
\cat{A}_\rightlocal\arrow[hook]{r}
&\cat{A}\kokoni{,}
\end{tikzcd}\]
where the vertical arrows are the forgetful functors, is Cartesian,
and the localization functor $\Mod_A\to\Mod_{A,\rightlocal}$ is the
functor induced from the localization functor
$\blank_\rightlocal\colon\cat{A}\to\cat{A}_\rightlocal$, and its lax
linearity over the action of $A$.
In particular, the localization functor lifts $\blank_\rightlocal$
canonically.

It follows that the complementary left localization
$\Mod_{A,\leftlocal}$ of $\Mod_A$ is given by the Cartesian square
\[\begin{tikzcd}
\Mod_{A,\leftlocal}\arrow[hook]{r}\arrow{d}
&\Mod_A\arrow{d}\\
\cat{A}_\leftlocal\arrow[hook]{r}
&\cat{A}\kokoni{.}
\end{tikzcd}\]
As a full subcategory of $\Mod_A$, this can be also expressed as
$\Mod_A(\cat{A}_\leftlocal)$, modules with respect to the op-lax
action of powers of $A$ on $\cat{A}_\leftlocal$ by $X\mapsto
X\tensor_\leftlocal A^{\tensor n}:=(X\tensor A^{\tensor
  n})_\leftlocal$.

\begin{remark}
The action of powers of $A$ on $\cat{A}_\leftlocal$ is in fact
genuinely associative.
To see this, it suffices to show that for any object $X$ of $\cat{A}$,
the map $X\tensor_\leftlocal A^{\tensor n}\to
X_\leftlocal\tensor_\leftlocal A^{\tensor n}$ is an equivalence.
This follows from the cofibre sequence
\[
X_\rightlocal\tensor A^{\tensor n}\longto X\tensor A^{\tensor
  n}\longto X_\leftlocal\tensor A^{\tensor n}
\]
and Lemma \ref{lem:local-equivalence}.
\end{remark}

The left localization $\Mod_A\to\Mod_{A,\leftlocal}$ lifts
$\blank_\leftlocal\colon\cat{A}\to\cat{A}_\leftlocal$ canonically,
since the left localization functor is the cofibre of the right
localization map, and the forgetful functor $\Mod_A\to\cat{A}$
preserves cofibre sequences.

In terms of objects, if $K$ is an $A$-module in $\cat{A}$, then
$K_\leftlocal$ has a canonical structure of an $A$-module.
This comes from the canonical structure of an $A$-module on
$K_\rightlocal$, and the canonical structure of an $A$-module map on
the right localization map $K_\rightlocal\to K$ (which \emph{together}
exist uniquely).

\begin{remark}
In particular, $A_\rightlocal$ is an $A$-bimodule, and hence
$A_\leftlocal$ becomes an $A$-algebra.
However, the $A$-module $K_\leftlocal$ does not in general come from
an $A_\leftlocal$-module.
\end{remark}

\subsubsection{}

Let us now consider a filtered stable category $\cat{A}$ with
compatible monoidal structure, and let $A$ be an associative
algebra in $\cat{A}$.
A sufficient condition so the constructions above can be applied to
this context is that the underlying object of $A$ belong to
$\cat{A}_{\ge 0}$.

Thus, let $A$ be in fact, an associative algebra in $\cat{A}_{\ge 0}$.
Then we have a filtration on $\Mod_A$, where $\Mod_{A,\ge
  r}=\Mod_A(\cat{A}_{\ge r})\sub\Mod_A$, and the localization functor
$\Mod_A\to\Mod_{A,\ge r}$ is induced from $\blank_{\ge
  r}\colon\cat{A}\to\cat{A}_{\ge r}$.

The complementary left localization also lifts that on $\cat{A}$, and
the square
\[\begin{tikzcd}
\Mod_A^{<r}\arrow[hook]{r}\arrow{d}
&\Mod_A\arrow{d}\\
\cat{A}^{<r}\arrow[hook]{r}
&\cat{A}
\end{tikzcd}\]
is Cartesian for every $r$.
As a full subcategory of $\Mod_A$, this can be also expressed as
$\Mod_A(\cat{A}^{<r})$, modules with respect to the action of $A$ on
$\cat{A}^{<r}$ by $X\mapsto(X\tensor A)^{<r}$.

The localization functor $\Mod_A\to\Mod_A^{<r}$ lifts
$\blank^{<r}\colon\cat{A}\to\cat{A}^{<r}$.

\begin{remark}
As noted in the previous remark, the $A$-module $K^{<r}$ does not in
general come from an $A^{<r}$-module.
However, it is always true that $\cat{A}_{\ge
  0}^{<r}:=\cat{A}^{<r}\intersect\cat{A}_{\ge 0}$ comes with a canonical
monoidal structure, together with a canonical
monoidal structure on the functor $\cat{A}_{\ge 0}\to\cat{A}_{\ge
  0}^{<r}$.
Note that the algebra $A^{<r}$ is obtained in $\cat{A}_{\ge 0}^{<r}$
using this.
If $A$-module $K$ is in $\cat{A}_{\ge 0}$, then $K^{<r}$ can be
obtained as an $A^{<r}$-module in $\cat{A}_{\ge 0}^{<r}$ from which the
structure of an $A$-module on $K^{<r}$ gets recovered.
\end{remark}

\subsubsection{}

If further, $\complete{\cat{A}}$ is the completion of $\cat{A}$, then
$\complete\Mod_A$ is the completion of $\Mod_A$, and this completion
lifts the completion of $\cat{A}$.
As a full subcategory of $\Mod_A$,
$\complete\Mod_A=\Mod_A\times_\cat{A}\complete{\cat{A}}$.

\begin{corollary}
$\Mod_A$ is complete if $\cat{A}$ is complete.
\end{corollary}

The full subcategory $\cat{A}_{\ge\infty}$ of $\cat{A}$ is preserved
by the action of $A$, so
the general argument can be applied to completion as well.
In particular, $\complete\Mod_A$ can be identified with
$\Mod_A(\complete{\cat{A}})$, where $A$ acts on $\complete{\cat{A}}$
by $X\mapsto X\complete\tensor A:=\complete{X\tensor A}$.
The inclusion $\complete\Mod_A\into\Mod_A$ then gets identified with
the functor induced from the lax $A$-linear functor
$\complete{\cat{A}}\into\cat{A}$.

If further, the monoidal structure on $\cat{A}$ is completable, then
the action of $A$ on $\complete{\cat{A}}$ is through the action of the
algebra $\complete{A}$ in $\complete{\cat{A}}$ (indeed we shall have
$X\tensor A\equivto X\tensor\complete{A}$) on
$\complete{\cat{A}}$, and the completion functor
\[
\Mod_A(\cat{A})\longto\complete\Mod_A(\cat{A})\equivwith\Mod_A(\complete{\cat{A}})=\Mod_{\complete{A}}(\complete{\cat{A}})
\]
is just the functor induced from the \emph{monoidal} functor
$\complete\blank\colon\cat{A}\to\complete{\cat{A}}$.

\subsubsection{}

Let $\cat{A}$ be a monoidal filtered stable
category, and let $A$ be an associative algebra in $\cat{A}_{\ge 0}$.
Then we consider the pairing
\[
-\tensor_A-\colon\Mod_A\times{}_A\Mod\longto\cat{A},
\]
where ${}_A\Mod$ denotes the filtered stable category of left
$A$-modules.
This pairing is compatible with the filtrations.

\begin{corollary}\label{cor:continuity-of-tesor-product}
Let $\cat{A}$ be a monoidal filtered stable category with
uniformly bounded sequential limits.

Let $A$ be an associative algebra in $\cat{A}_{\ge 0}$.
Suppose given an inverse system
\[
\cdots\longfrom K_i\longfrom K_{i+1}\longfrom\cdots
\]
in $\Mod_A$, and suppose there is a sequence
$(r_i)_i$ of integers, tending to $\infty$ as $i\to\infty$,
such that for every $i$, the fibre of the map $K_{i+1}\to K_i$ belongs
to $\Mod_{A,\ge r_i}$ (namely, its underlying object belongs to
$\cat{A}_{\ge r_i}$).
Then for every left $A$-module $L$, if (the underlying object
of) $L$ is bounded below, then the induced map
\[
(\lim_iK_i)\tensor_AL\longto\lim_i(K_i\tensor_AL)
\]
is an equivalence after completion.
\end{corollary}

\subsubsection{}

We discuss a few simple consequences.
(More consequences will be discussed in the next section.)

Firstly, associativity of tensor product holds for bounded below
modules over \emph{positive} augmented algebras to be defined as
follows.

\begin{definition}
Let $\cat{A}$ be a monoidal filtered stable category.
We say that an augmented algebra $A$ in $\cat{A}$ is \kore{positive}
if the augmentation ideal $I$ of $A$ belongs to $\cat{A}_{\ge 1}$.
\end{definition}

\begin{lemma}\label{lem:associativity-restricted}
Let $\cat{A}$ be a monoidal complete filtered stable category.
Let $A_i$, $i=0,1,2,3$, be positive augmented algebras in $\cat{A}$,
and let $K_{i,i+1}$ be a left $A_i$- right $A_{i+1}$-bimodule for
$i=0,1,2$, whose underlying object is bounded below.

Then the resulting map
\[
K_{01}\tensor_{A_1}K_{12}\tensor_{A_2}K_{23}\longto(K_{01}\tensor_{A_1}K_{12})\tensor_{A_2}K_{23}
\]
is an equivalence, where the source denotes the realization of the
\textbf{bi}simplicial bar construction (each simplicial index coming
from the actions of eachof the algebras $A_1$, $A_2$).
\end{lemma}
\begin{proof}
Denote the augmentation ideal of $A_i$ by $I_i$.
We express the tensor product $K_{01}\tensor_{A_1}K_{12}$ etc.~as the
geometric realization of the simplicial bar construction
$B_\dot(K_{01},I_1,K_{12})$ etc.~\emph{without} degeneracies (in the
sense that it is a diagram over $\Simpface$), associated to the
actions of the non-unital algebra $I_1$ etc.
See Section \ref{sec:totalization}.
It is easy to check that the usual bar construction, with
degeneracies, associated to the unital algebra $A_1$ etc., is the left
Kan extension of the version here, so the geometric realizations are
equivalent.

The target then can be written as
$|B_\dot(K_{01}\tensor_{A_1}K_{12},I_2,K_{23})|$.

For every $n$, the functor $-\tensor I_2^{\tensor n}\tensor K_{23}$ is
bounded below, so Proposition \ref{prop:realization-preserved} implies
that
\[
B_\dot(K_{01}\tensor_{A_1}K_{12},I_2,K_{23})=|B_\dot(B_*(K_{01},I_1,K_{12}),I_2,K_{23})|,
\]
where the realization is in the variable $*$.

However, the realization of this is nothing but the source.
\end{proof}

Let $\cat{A}$ be a monoidal complete filtered stable category, and let
$A$ be a positive augmented associative algebra in $\cat{A}$.
Let $\varepsilon\colon A\to\unity$ be the augmentation map, and
$I:=\Fibre(\varepsilon)$ be the augmentation ideal of $A$.

Let us define the \emph{powers} of $I$ by $I^r:=I^{\tensor_Ar}$.
Note that multiplication of $A$ gives an $A$-bimodule map $I^r\to I^s$
whenever $r\ge s$.
Denote the cofibre of this map by $I^s/I^r$.
When $s=0$, this, $A/I^r$, is an $A$-algebra.

\begin{lemma}\label{lem:completeness-of-module}
Let $\cat{A}$ be a monoidal filtered stable category with
$\complete{\cat{A}}$ completing it, and let $A$ be a positive
augmented associative algebra in $\cat{A}$.

Let $K$ be a right $A$-module which is bounded below.
Then the map $K\to\lim_rK\tensor_AA/I^r$ is an equivalence after
completion.
\end{lemma}
\begin{proof}
Since the fibre of the map $A\to A/I^r$ (namely $I^r$) belongs
to $\Mod_{A,\ge r}$, the result follows from
Corollary~\ref{cor:equivalence-of-limits-after-completion}.
(Write $K$ as $K\tensor_AA$.)
\end{proof}

\begin{corollary}\label{cor:josh-completeness}
Let $A$ be a positive augmented associative algebra in a
monoidal complete filtered stable category $\cat{A}$.
Then, the functor $-\tensor_A\unity\colon\Mod_{A,\ge r}\to\cat{A}_{\ge
  r}$ reflects equivalences.
\end{corollary}
\begin{proof}
Suppose an $A$-module $K$ in $\cat{A}_{\ge r}$ satisfies
$K\tensor_A\unity\equivwith\zero$.
We want to show that $K\equivwith\zero$.

In order to do this, it suffices, from the previous lemma, to
prove $K\tensor_A(I^s/I^{s+1})\equivwith\zero$ for all $s\ge 0$.
However, $I^s/I^{s+1}\equivwith\unity\tensor_AI^s$ as a left
$A$-module.
\end{proof}

\section{Koszul duality for complete algebras}
\label{sec:complete-poincare}
\setcounter{subsection}{-1}
\setcounter{equation}{-1}

\subsection{Introduction}

In this section, we shall obtain our main results on the Koszul
duality using the basic results developed in the previous two
sections.

\subsection{Koszul completeness of a positive algebra}
\label{sec:koszul-complete-algebra}
\setcounter{subsubsection}{-1}

\subsubsection{}
 
The Koszul duality we consider will be between augmented algebras and
coalgebras.
We first need to consider the condition on an augmented coalgebra,
corresponding to the positivity of an algebra.

\begin{definition}
Let $\cat{A}$ be a monoidal filtered stable category with uniformly
bounded loops and sequential limits.
An augmented associative coalgebra $C$ in $\cat{A}$
is said to be \kore{copositive} (with respect to the filtration)
if the augmentation ideal $J$ belongs to $\cat{A}_{\ge 1-\omega}$ for
a uniform bound $\omega$ for loops in $\cat{A}$.
\end{definition}

\begin{example}
If the filtration is a t-structure, then copositivity means that
$\Omega J$ belongs to $\cat{A}_{\ge 1}$.
\end{example}

Let us now consider the Koszul duality.
For an augmented coalgebra $C$, recall that its \emph{Koszul dual} is
an augmented associative algebra $C^!$ described as follows.

First of all, its underlying object is $\unity\cotensor_C\unity$,
where $\unity$ is given the structure of a $C$-module coming through
the augmentation map $\varepsilon\colon\unity\to C$ from the module
structure of $\unity$ over the unit coalgebra, and $\cotensor_C$
denotes the cotensor product operation over $C$.

In other words, it is an object representing the presheaf
$\cat{A}^\op\to\Space$, $X\mapsto\Map_{\Mod_C}(X\tensor\unity,
\unity)$, where $X\tensor\unity$($=X$) is made into a
$C$-module by the action of $C$ on the factor $\unity$.
The structure of an associative algebra of $C^!$ results from this,
and we take as the augmentation the map $\eta^!\colon
C^!\to\unity^!=\unity$ for the unit $\eta\colon C\to\unity$.

From this description, $C^!$ represents the presheaf on the category
of augmented associative algebras which maps an object $A$ to the
space of $A$-module structures on the $C$-module $\unity$, lifting the
$A$-module structure on the underlying object $\unity$ given by the
augmentation map of $A$.
In particular, $\Map_{\Alg_*}(A,C^!)=\Map_{\Coalg_*}(A^!,C)$, where
$A^!=\unity\tensor_A\unity$ is the augmeted associative coalgebra
Koszul dual to $A$.
The  subscripts $*$ here indicates that the categories are those of
\emph{augmented} algebras and coalgebras.
(For example, the map
$A^!\xrightarrow{\eta}\unity\xrightarrow{\varepsilon}C$
corresponds to the map
$A\xrightarrow{\varepsilon}\unity\xrightarrow{\eta}C^!$.)

Let $\cat{A}$ be a monoidal filtered stable category with uniformly
bounded loops and sequential limits.
Then the following lemma implies for a copositive augmented
associative coalgebra $C$ in $\cat{A}$, that its Koszul dual algebra
is positive.
We formulate the lemma more general than needed here, so it can be
iterated for a later use.

\begin{lemma}\label{lem:positive-koszul-dual}
Let $\cat{A}$ be a monoidal filtered stable category with uniformly
bounded loops and sequential limits.
Let $C$ be an augmented associative coalgebra in $\cat{A}$, and let
$K_i$, $i=0,1$, be a right and a left $C$-modules respectively,
equipped with maps $\eta_i\colon K_i\to\unity$ of $C$-modules.
Assume that, for an integer $r\ge 1$ and a uniform bound $\omega$ for
loops in $\cat{A}$, the augmentation ideal of $C$ belongs to the full
subcategory $\cat{A}_{\ge r-\omega}$ of $\cat{A}$, and $\Fibre\eta_i$
belongs to $\cat{A}_{\ge r}$ for $i=0,1$.

Let $\eta\colon C\to\unity$ be the unit map of $C$.
Then the fibre of the map
\[
K_0\cotensor_CK_1\xlongrightarrow{\eta_0\tensor_\eta\eta_1}\unity\cotensor_\unity\unity=\unity
\]
belongs to the full subcategory $\cat{A}_{\ge r}$ of $\cat{A}$.
\end{lemma}
\begin{proof}
Let $J$ be the augmentation ideal of $C$, so $J\in\cat{A}_{\ge
  r-\omega}$.
Since we have Corollary~\ref{cor:closure-under-extension}, it
suffices to prove that the fibre of each of the following obvious
maps belongs to $\cat{A}_{\ge r}$:
\begin{align*}
  K_0\cotensor_CK_1=\Tot B^\dot(K_0,J,K_1)
  &\longto\sk_{-d}\Tot B^\dot(K_0,J,K_1)\\
  &\longto\sk_0\Tot B^\dot(K_0,J,K_1)=K_0\tensor K_1\\
  &\xlongrightarrow{\eta_0\tensor\eta_1}\unity\tensor\unity=\unity,
\end{align*}
where $d\le 0$ is a uniform lower bound for sequential limits in
$\cat{A}$.
We shall prove
\begin{equation}\label{eq:fibre-over-skeleton}
  \Fibre[\Tot B^\dot(K_0,J,K_1)\to\sk_{-d}\Tot B^\dot(K_0,J,K_1)]
  \in\cat{A}_{\ge r}. 
\end{equation}
The rest is either similar or simpler.

In order to prove \eqref{eq:fibre-over-skeleton}, by the
definition~\ref{def:bounded-limit} of a uniform lower bound for
sequential limits, it suffices to prove that the fibre of the map
\[
  \sk_n\Tot B^\dot(K_0,J,K_1)\longto\sk_{-d}\Tot B^\dot(K_0,J,K_1)
\]
belongs to $\cat{A}_{\ge r-d}$ for all $n\ge-d+1$.
However, this follows from
Corollary~\ref{cor:closure-under-extension}, since for every
$k\ge-d+1$, the fibre $\Loop^kB^k(K_0,J,K_1)=\Loop^kK_0\tensor
J^{\tensor k}\tensor K_1$ of the map $\sk_k\Tot
B^\dot(K_0,J,K_1)\to\sk_{k-1}\Tot B^\dot(K_0,J,K_1)$ belongs to
$\cat{A}_{\ge kr}\sub\cat{A}_{\ge r-d}$.
\end{proof}

\begin{definition}\label{def:translated-filtration}
Let $\cat{A}$ be a filtered stable category.
Then we say that looping \kore{translates the filtration} of
$\cat{A}$, or is \kore{translational} in $\cat{A}$, if there is a
uniform lower bound $\omega$ for loops in $\cat{A}$ for which
$-\omega$ is a lower bound of the functor
$\susp=\Loop^{-1}\colon\cat{A}\to\cat{A}$.
Equivalently (by Lemma~\ref{lem:adjoint-preserving-localization}),
$\omega$ which is also an upper bound of the functor $\Loop$.
\end{definition}

\begin{remark}\label{rem:sound-example}
This happens if the tensoring of the finite spectra on $\cat{A}$ is
compatible with the filtrations.
See Remark \ref{rem:on-bounded-loops}.

Examples include the category of filtered objects
(Section \ref{sec:examples-of-filtered-categories}), a stable category
with a t-structure (Example \ref{ex:t-structure}), and a functor
category with Goodwillie's filtration
(Example \ref{ex:goodwillie-filtration}).
\end{remark}

\begin{remark}
In general, if looping translates the filtration of $\cat{A}$, and if
there exists an integer $r$ for which $\cat{A}_{\ge r+1}$ is a
\emph{proper} subcategory of $\cat{A}_{\ge r}$, and equivalently,
$\cat{A}^{<r}$ is a proper subcategory of $\cat{A}^{<r+1}$, then a
lower bound $\omega$ of $\Loop\colon\cat{A}\to\cat{A}$ for which
$-\omega$ is an upper bound of $\susp\colon\cat{A}\to\cat{A}$, must
be the greatest lower bound of $\Loop$, as well as the least upper
bound of $\Loop$ by duality.
\end{remark}

The following proposition might be clarifying.

\begin{proposition}
Let $\cat{A}$ be a filtered stable category.
Then an integer $\omega$ is a lower and an upper bound of the functor
$\Loop\colon\cat{A}\to\cat{A}$ if and only if for every integer $r$
and every object $X\in\cat{A}$, we have an equivalence
$(\Loop X)^{<r+\omega}\equivwith\Loop(X^{<r})$ in $\cat{A}$.
\end{proposition}
\begin{proof}
If $\omega$ is a lower and upper bound of $\Loop$, then, since in the
cofibre sequence
\[
  \Loop(X_{\ge r})\longto\Loop X\longto\Loop(X^{<r}),
\]
the fibre and the cofibre will respectively be in $\cat{A}_{\ge
  r+\omega}$ and be in $\cat{A}^{<r+\omega}$, we have that the map
$\Loop X\to\Loop(X^{<r})$ induces an equivalence $(\Loop
X)^{<r+\omega}\equivto\Loop(X^{<r})$ by
Corollary~\ref{cor:characterizing-localizations}.

Conversely, suppose we have equivalences $(\Loop
X)^{<r+\omega}\equivwith\Loop(X^{<r})$.
Then for $X\in\cat{A}$ belonging to $\cat{A}^{<r}$, this implies that
$\Loop X=\Loop(X^{<r})$ belongs to $\cat{A}^{<r+\omega}$, so $\Loop$
takes the full subcategory $\cat{A}^{<r}$ of $\cat{A}$ to the full
subcategory $\cat{A}^{<r+\omega}$.
For $X$ belonging to $\cat{A}_{\ge r}$, we obtain $(\Loop
X)^{<r+\omega}\equivwith\Loop(X^{<r})\equivwith\zero$, so $\Loop$
takes the full subcategory $\cat{A}_{\ge r}$ to $\cat{A}_{\ge
  r+\omega}$.
\end{proof}

\begin{lemma}\label{lem:copositive-koszul-dual}
Let $\cat{A}$ be a monoidal soundly filtered stable category.
If $A$ is a positive augmented associative algebra in $\cat{A}$, then
its Koszul dual coalgebra is copositive.
\end{lemma}
\begin{proof}
Similar to the proof of Lemma \ref{lem:positive-koszul-dual}, but is
simpler.
\end{proof}

\begin{proposition}\label{prop:commuting-tensor-and-cotensor}
Let $\cat{A}$ be a monoidal complete filtered stable category with
uniformly bounded loops and sequential limits.

Let $A$ be a positive augmented associative algebra, and $C$ a
copositive augmented associative coalgebra, both in $\cat{A}$.
Let $K$ be a right $A$-module, $L$ an $A$--$C$-bimodule, and let $X$
be a left $C$-module, all bounded below.

Then the canonical map
\[
K\tensor_A(L\cotensor_CX)\longto(K\tensor_AL)\cotensor_CX
\]
is an equivalence (where the left $A$-module structure of
$L\cotensor_CX$ and the right $C$-module structure of $K\tensor_AL$
are induced from the $A$--$C$-bimodule structure of $L$).
\end{proposition}
\begin{proof}
Write
\[
L\cotensor_CX=\Tot B^\dot(L,J,X).
\]
Since the functor $K\tensor_A-$ is bounded below, we obtain from
Proposition \ref{prop:tatalization-preserved} that this functor sends
this cotensor product to $\Tot K\tensor_AB^\dot(L,J,X)$.

Since $J$ and $X$ are bounded below, it follows from Proposition
\ref{prop:realization-preserved} that
\[
K\tensor_AB^\dot(L,J,X)=B^\dot(K\tensor_AL,J,X).
\]
Therefore, we get the result by totalizing this.
\end{proof}

\subsubsection{}

Let $A$ be an augmented associative algebra.
Then, for a right $A$-module $K$, we define a right $A^!$-module
$\Dual_AK$ as $K\tensor_A\unity$.
Dually, if $C$ is an augmented associative coalgebra, then for a right
$C$-module $L$, we have a right $C^!$-module
$\Dual_CL=L\cotensor_C\unity$.

If $K$ is a \emph{left} $A$-module, then we simply define a left
$A^!$-module $\Dual_AK$ by $\unity\tensor_AK$, and similarly for left
$C$-modules.

The following is a special case of
Proposition \ref{prop:commuting-tensor-and-cotensor}.
\begin{corollary}
Let $\cat{A}$ be a monoidal complete soundly filtered stable
category with uniformly bounded sequential limits.
Let $A$ be an augmented associative algebra in $\cat{A}$, and assume
it is positive.
Let $K$ be a right $A$-module, $L$ a left $A^!$-module, and assume
both of these are bounded below.

Then the canonical map
\[
K\tensor_A\Dual_{A^!}L\longto\Dual_AK\cotensor_{A^!}L
\]
is an equivalence.
\end{corollary}
\begin{proof}
The coalgebra $A^!$ is copositive by
Lemma \ref{lem:copositive-koszul-dual}.
\end{proof}

\begin{corollary}\label{cor:twice-dual-of-comodule}
In the situation of the previous corollary, the canonical map
$\Dual_A\Dual_{A^!}L\to L$ is an equivalence.
\end{corollary}
\begin{proof}
Apply the previous corollary to $K=\unity$.
\end{proof}

\begin{theorem}\label{thm:josh-complete-implies-koszul-complete}
Let $\cat{A}$ be a monoidal complete soundly filtered stable
category with uniformly bounded sequential limits.
Let $A$ be a positive augmented associative algebra in $\cat{A}$, and
$K$ be a right $A$-module which is bounded below.
Then the canonical map $K\to\Dual_{A^!}\Dual_AK$ is an equivalence (of
$A$-modules).
In particular, the canonical map $A\to A^{!!}$ (of augmented
associative algebras) is an equivalence.
\end{theorem}
\begin{proof}
By Corollary \ref{cor:josh-completeness}, it suffices to prove that
the map is an equivalence after we apply the functor
$-\tensor_A\unity$ to it.
However, this follows by applying
Corollary \ref{cor:twice-dual-of-comodule} to the (right) $A^!$-module
$\Dual_AK$.
\end{proof}

We remark that the proof in fact proves the following.

\begin{lemma}\label{lem:josh-complete-implies-koszul-complete-refine}
Let $\cat{A}$ be a monoidal complete filtered stable
category with uniformly bounded loops and sequential limits.
If $A$ is a positive augmented associative algebra in $\cat{A}$ such
that the augmented associative coalgebra $A^!$ is copositive, then
the conclusion of
Theorem~\ref{thm:josh-complete-implies-koszul-complete} holds.
\end{lemma}

\subsection{Koszul completeness of a coalgebra}
\setcounter{subsubsection}{-1}

In order to complete our study of the Koszul duality for associative
algebras, we shall establish the results similar to those established
for positive augmented algebras, for copositive coalgebras.

Let us start with the following situation.
Namely, let $C_i$, $i=0,1,2$, be coalgebras in $\cat{A}$,
and let $K_{i,i+1}$ for $i=0,1$ be a left $C_i$- right
$C_{i+1}$-bimodule.
Then we would like $K_{01}\cotensor_{C_1}K_{12}$ to be a
$C_0$--$C_2$-bimodule in a natural way.

We have this in the following case.
Namely, assume $\cat{A}$ to be a monoidal complete filtered
stable category with uniformly bounded loops and sequential limits.
Moreover, assume that $C_1$ is a copositive augmented coalgebra, and
$K_{i,i+1}$ are bounded below.
Then for any bounded below object $L$, the canonical map
\[
(K_{01}\cotensor_{C_1}K_{12})\tensor L\longto
K_{01}\cotensor_{C_1}(K_{12}\tensor L)
\]
is an equivalence by Proposition \ref{prop:tatalization-preserved}.

It follows that if $C_0$ and $C_2$ are bounded below, then the
bimodule structures of $K_{i,i+1}$, $i=0,1$ induce a structure of a
$C_0$--$C_2$-bimodule on the cotensor product.
In fact, the resulting bimodule has the universal property to be
expected of the cotensor product.

\begin{lemma}\label{lem:cotensor-associative}
Let $\cat{A}$ be a monoidal complete filtered
stable category with uniformly bounded loops and sequential limits.
Let $C_i$, $i=0,1,2,3$, be copositive augmented coalgebras in
$\cat{A}$, and let $K_{i,i+1}$ be a left $C_i$- right
$C_{i+1}$-bimodule for $i=0,1,2$, whose underlying object is bounded
below.

Then the resulting map
\[
(K_{01}\cotensor_{C_1}K_{12})\cotensor_{C_2}K_{23}\longto
K_{01}\cotensor_{C_1}K_{12}\cotensor_{C_2}K_{23}
\]
is an equivalence of $C_0$--$C_3$-modules, where the target denotes
the totalization of the \textbf{bi}cosimplicial bar construction
(dual to the corresponding construction in
Lemma~\ref{lem:associativity-restricted}).
\end{lemma}

The proof is similar to the proof of Lemma
\ref{lem:associativity-restricted}.
One uses Proposition \ref{prop:tatalization-preserved} instead of
Proposition \ref{prop:realization-preserved}.

The proof of the following lemma is similar to the proof of
Lemma~\ref{lem:positive-koszul-dual}, but is simpler.

\begin{lemma}\label{lem:cotensor-bounded-below}
Let $\cat{A}$ be a monoidal filtered stable category with uniformly
bounded loops and sequential limits.
Let $C$ be a copositive augmented coalgebra, $K$ a right
$C$-module, and $L$ a left $C$-module, all in $\cat{A}$.
If for integers $r$ and $s$, (the underlying object of) $K$ belongs to
$\cat{A}_{\ge r}$, and $L$ belongs to $\cat{A}_{\ge s}$, then
$K\cotensor_CL$ belongs to $\cat{A}_{\ge r+s}$.
\end{lemma}

Let $\Mod_{C,>-\infty}$ denote the category of bounded below right
$C$-modules.
\begin{lemma}\label{lem:josh-completeness-over-coalgebra}
Let $\cat{A}$ be a monoidal complete filtered
stable category with uniformly bounded loops and sequential limits.
Let $C$ be a copositive augmented coalgebra in $\cat{A}$.
Then the functor
\[
-\cotensor_C\unity\colon\Mod_{C,>-\infty}\to\cat{A}
\]
reflects equivalences.
\end{lemma}
\begin{proof}
We would like to apply the arguments of the proof of
Corollary \ref{cor:josh-completeness}.
We simply need to establish the counterpart of
Lemma~\ref{lem:completeness-of-module}.
This follows from Lemma \ref{lem:cotensor-bounded-below}.
\end{proof}

\begin{lemma}\label{lem:koszul-complete-coalgebra}
Let $C$ be as in Lemma \ref{lem:josh-completeness-over-coalgebra}, and
let $K$ be a right $C$-module which is bounded below.

Then the canonical map $\Dual_{C^!}\Dual_CK\to K$ is an equivalence
(of $C$-modules).
In particular, the canonical map $C^{!!}\to C$ (of augmented
associative coalgebras) is an equivalence.
\end{lemma}
\begin{proof}
Similar to the proof of
Theorem \ref{thm:josh-complete-implies-koszul-complete}.
Note that the assumptions imply that $C^!$ is positive, so
Proposition \ref{prop:commuting-tensor-and-cotensor} can be applied.
\end{proof}

\begin{theorem}\label{thm:koszul-equivalence-of-modules}
Let $\cat{A}$ be a monoidal complete filtered
stable category with uniformly bounded loops and sequential limits.
Let $C$ be a copositive augmented coalgebra in $\cat{A}$.
Then the functor
\[
\Dual_C\colon\Mod_{C,>-\infty}\longto\Mod_{C^!,>-\infty}
\]
is an equivalence with inverse $\Dual_{C^!}$.
\end{theorem}
\begin{proof}
This follows from Lemmata~\ref{lem:koszul-complete-coalgebra} and
\ref{lem:josh-complete-implies-koszul-complete-refine}.
\end{proof}

From Lemma \ref{lem:koszul-complete-coalgebra} and Theorem
\ref{thm:josh-complete-implies-koszul-complete}, we also obtain
immediately the case $n=1$ of Theorem
\ref{thm:associative-koszul-duality} below.

\subsection{Koszul duality for $E_n$-algebras}
\label{sec:koszul-duality-e_n}
\setcounter{subsubsection}{-1}

In this section, we would like to prove our first main theorem, which
extracts an equivalence of categories from the Koszul duality.
In this section, we assume that $\cat{A}$ is a monoidal complete
soundly filtered stable category with uniformly bounded sequential
limits.

We define the Koszul duality functor for $E_n$-algebras inductively as
the composite
\[
\Alg_{E_1*}(\Alg_{E_{n-1}*})\longto\Coalg_{E_1*}(\Alg_{E_{n-1}*})\longto\Coalg_{E_1*}(\Coalg_{E_{n-1}*}),
\]
where the first map is the associative Koszul duality construction,
and the next map is induced from the inductively defined
$E_{n-1}$-Koszul duality functor, which is canonically op-lax
symmetric monoidal by induction.

We would like to analyse this for a suitable restricted classes of
algebras and coalgebras (considered as algebras in the opposite
category).
The restriction will be given by some positivity conditions as below.

\begin{definition}\label{def:e-n-copositive}
Let $\cat{A}$ be a symmetric monoidal filtered stable category.

An augmented $E_n$-algebra $A$ is said to be \kore{positive} if its
augmentation ideal belongs to $\cat{A}_{\ge 1}$.

An augmented $E_n$-coalgebra $C$ in $\cat{A}$ is said to be
\kore{copositive} if there is a uniform lower bound $\omega$ for
loops in
$\cat{A}$ such that the augmentation ideal of $C$ belongs to
$\cat{A}_{\ge 1-n\omega}$.
\end{definition}

In addition to easily found examples of positive $E_n$-algebras in
particular filtered categories of the kinds named in Introduction
(Section \ref{sec:introduction}), there is also a manner in which a
positive $E_n$-algebra and more generally, a (locally constant)
factorization algebra, arises from \emph{any} augmented ($E_n$- or
factorization) algebra in a reasonable (non-filtered) symmetric
monoidal stable category.
Namely, such an algebra can be naturally filtered by certain `powers'
of its augmentation ideal, to give rise to a positive augmented
algebras in the category of filtered objects described in Section
\ref{sec:examples-of-filtered-categories}.
We refer the reader to \cite{poincare} for the details.

\begin{lemma}\label{lem:monoidality-of-koszul-duality}
Let $A$, $B$ be positive augmented associative algebras in a
symmetric monoidal complete filtered stable category $\cat{A}$.
Then the canonical map $(A\tensor B)^!\to A^!\tensor B^!$ is an
equivalence.
\end{lemma}
\begin{proof}
This follows from Lemma \ref{lem:associativity-restricted}.
\end{proof}

In other words, the functor $A\mapsto A^!$ is symmetric monoidal, so
in particular, if $A$ is a positive augemented $E_{n+1}$-algebra,
then the Koszul dual $\unity\tensor_A\unity$ of its underlying
associative algebra becomes an $E_n$-algebra in the category of
augmeted associative coalgebras.
Moreover, by Proposition \ref{prop:realization-preserved}, this
$E_n$-algebra is equivalent to the tensor product
$\unity\tensor_A\unity$ \emph{taken in the category of
  $E_n$-algebras}.

By Lemma \ref{lem:cotensor-associative}, similar results
hold for copositive $E_n$-coalgebras as well.
It follows that Lemmas similar to Lemma \ref{lem:positive-koszul-dual}
and \ref{lem:copositive-koszul-dual} holds for $E_n$-algebras.

We shall now state our first main theorem.
Let $\Alg_{E_n}(\cat{A})_\positive$ denote the category of
positive augmented $E_n$-algebras in $\cat{A}$, and similarly,
$\Coalg_\positive$ for copositive coalgebras.
\begin{theorem}\label{thm:associative-koszul-duality}
Let $\cat{A}$ be a monoidal complete soundly filtered stable
category with uniformly bounded sequential limits.
Then the constructions of Koszul duals give inverse equivalences
\[
\Alg_{E_n}(\cat{A})_\positive\xlongleftrightarrow{\equiv}\Coalg_{E_n}(\cat{A})_\positive.
\]
\end{theorem}
\begin{proof}
The proof will be by induction on $n$.
We claim the equivalence as well as the following.
Namely, under the claimed equivalence, if $A\in\Alg_{E_n+}$ and
$C\in\Coalg_{E_n+}$ correspond to each other, then we further
claim that for every integer $r$, the condition that the augmentation
ideal $I$ of $A$ belongs to $\cat{A}_{\ge r}$, is equivalent to that
the augmentation ideal $J$ of $C$ belongs to $\cat{A}_{\ge
  r-n\omega}$ for the uniform bound $\omega$ for loops in $\cat{A}$
satisfying the condition stated in Definition
\ref{def:translated-filtration}.
We prove these claims by induction on $n$.

The case $n=0$ is obvious, so assume the claims for an integer $n\ge
0$.
Then we would like to prove the claims for $n+1$.

With the preparation above, the arguments of the previous sections
apply, under a modification, to augmented associative algebras and
coalgebras in the category of $E_n$-algebras.
The modification needed is as follows.
Namely, the arguments refer to depth of objects in the filtration.
Since we are here dealing not with objects of $\cat{A}$, but with
$E_n$-algebras in $\cat{A}$, we should understood the depth of
algebras as the depth of the underlying objects in the filtration of
$\cat{A}$.
From this, we obtain that the constructions of the Koszul duals
restrict to an equivalence
\[
\Alg_{E_1}(\Alg_{E_n}(\cat{A})_+)_+\xlongleftrightarrow{\equiv}\Coalg_{E_1}(\Alg_{E_n}(\cat{A})_+)_+.
\]

On the other hand, from the inductive hypothesis, we have an
equivalence
\[
\Alg_{E_n}(\cat{A})_\positive \xlongleftrightarrow{\equiv}\Coalg_{E_n}(\cat{A})_\positive
\]
(which is symmetric monoidal by iteration of Lemma
\ref{lem:monoidality-of-koszul-duality}), in which the condition
$I\in\cat{A}_{\ge r}$ corresponds to $J\in\cat{A}_{\ge r-n\omega}$.

We obtain the desired equivalence for $n+1$ from these.
Moreover, suppose $A\in\Alg_{E_{n+1}+}$ and $C\in\Coalg_{E_{n+1}+}$
correspond to each other in the equvalence.
Then it follows from Lemmas \ref{lem:positive-koszul-dual} and
\ref{lem:copositive-koszul-dual} that the condition $I\in\cat{A}_{\ge
  r}$ is equivalent to that the augmentation ideal of the
associative Koszul dual of $A$ belongs to $\cat{A}_{\ge r-\omega}$.
Moreover, by the inductive hypothesis, this is equivalent to that
$J\in\cat{A}_{\ge r-(n+1)\omega}$.

This completes the inductive step.
\end{proof}

\subsection{Morita structure of the Koszul duality}
\label{sec:morita-structure}
\setcounter{subsubsection}{-1}

\subsubsection{}
 
Let $\cat{A}$ be a symmetric monoidal category whose monoidal
multiplication functor preserves geometric realizations
(variable-wise).
Then there is an ($n+1$)-category of $E_n$-algebras, generalizing the
Morita $2$-category of associative algebras.

If the condition of the preservation of geometric realization is
dropped, then one has to be cautious.
Let us work in $\cat{A}^\op$ instead, and see a case where a suitably
restricted higher dimensional Morita category of coalgebras in
$\cat{A}$ makes sense.

Specifically, for $\cat{A}$ a symmetric monoidal complete soundly
filtered stable
category with uniformly bounded sequential limits, we would like to
construct a version $\Coalg^\positive_n(\cat{A})$ of the Morita
$(n+1)$-category, of augmented coalgebras, in which $k$-morphisms are
copositive as an augmented $E_{n-k}$-coalgebra in $\cat{A}$.
The construction will be indeed the same as in (a version where
everything is augmented, of) the familiar case.
We shall simply observe that the usual construction makes sense under
the mentioned restriction on the objects to be morphisms in
$\Coalg^\positive_n(\cat{A})$.

Let us see why this is true.
We shall follow the construction outlined by Lurie in \cite{tft}.
Firstly, an object of $\Coalg^\positive_n(\cat{A})$ will be a
copositive augmented $E_n$-coalgebra in $\cat{A}$.
Given objects $C$, $D$ as such, then we would like to define the
morphism $n$-category $\Map(C,D)$ in $\Coalg^\positive_n(\cat{A})$ to
be what we shall denote by
$\Coalg^\positive_{n-1}(\Bimod_{D\naga C}(\cat{A}))$.
By this, we mean the Morita $n$-category to be seen to be
well-defined, for the $E_{n-1}$-monoidal category of
$C$--$D$-bimodules, in which ($k-1$)-morphisms are copositive as an
augmented $E_{n-k}$-coalgebra in $\cat{A}$.
(We understand an augmentation of an coalgebra in
$\Bimod_{C\naga D}(\cat{A})_{>-\infty}$ to be given by a map from
$\unity$, but \emph{not} from the unit of $\Bimod_{C\naga
  D}(\cat{A})_{>-\infty}$.)
For the moment, it will suffice to see that the cotensor product over
$C^\op\tensor D$ makes the category $\Bimod_{C\naga
  D}(\cat{A})_{>-\infty}$ of bounded below bimodules into an
$E_{n-1}$-monoidal category.
The reason why this will suffice is that it will be clear as we
proceed that the rest of the arguments for well-definedness of
$\Coalg^\positive_{n-1}(\Bimod_{C\naga D}(\cat{A}))$ is similar
(with obvious minor modifications) to the arguments for
well-definedness of $\Coalg^\positive_n$ we are discussing right now,
so we can ignore the issue of well-definedness of
$\Coalg^\positive_{n-1}(\Bimod_{C\naga D}(\cat{A}))$ for the moment
by understanding that the whole argument will be inductive at the end
(as in Lurie's description of the construction in the more familiar
case).

To investigate the cotensor product operation in $\Bimod_{C\naga
  D}(\cat{A})_{>-\infty}$, the associativity follows from
Proposition \ref{lem:cotensor-associative}.
If $n-1\ge 2$, we need to have compatibility of this operation with
itself.
This follows from the following general considerations.

\begin{lemma}\label{lem:cotensors-compatible}
Let $C$ be a copositive augmented $E_2$-coalgebra.
Let $C_i$, $D_i$, $i=0,1$, be associative coalgebras in the category
$(\Mod_C)_{\unity/}$ of augmented $C$-modules in $\cat{A}$.
Assume that these are copositive as an augmented associative coalgebra
in $\cat{A}$.

Let $D_{ij}$, $j=0,1$, be a bounded below $D_j$--$C_i$-bimodule if
$i+j$ is even, and a bounded below $C_i$--$D_j$-bimodule if $i+j$ is
odd.
Then the canonical map in $\cat{A}$ from
$(D_{00}\cotensor_{D_0^\op}D_{01})\cotensor_{C_0\cotensor_CC_1^\op}(D_{10}\cotensor_{D_1}D_{11})$
to the totalization of the corresponding bicosimplicial bar
construction is an equivalence.
\end{lemma}
\begin{proof}
Let us denote the augmentation ideal of $C$ and $C_i$ by $I$ and $I_i$
respectively, and the augmentation ideal of $D_j$ by $J_j$.
The bicosimplicial bar construction is
\[
B^\dot\bigl(B^\star(D_{00},J_0^\op,D_{01}),B^\star(I_0,I,I_1^\op),B^\star(D_{10},J_1,D_{11})\bigr),
\]
where the cosimplicial indices are $\dot$ and $\star$.
The result follows since the totalization of this in the index $\star$
is
\[
B^\dot(D_{00}\cotensor_{D_0^\op}D_{01},C_0\cotensor_CC_1^\op,D_{10}\cotensor_{D_1}D_{11})
\]
by Proposition \ref{prop:tatalization-preserved} and boundedness below
of the monoidal operations.
\end{proof}

\begin{lemma}
Let $C$ be an augmented associative coalgebra in $\cat{A}$, and let
$K$, $L$ be a right and a left $C$-modules respectively, both of which
are augmented.
Let $\varepsilon$ be the augmentation map of $C$, $K$, or $L$, and
assume that, for an integer $r$ and a uniform bound $\omega$ for loops
in $\cat{A}$, the cofibre in $\cat{A}$ of $\varepsilon$ belongs to
$\cat{A}_{\ge r-\omega}$ for $C$, and to $\cat{A}_{\ge r}$ for $K$ and
$L$.
Then the cofibre of the map
\[
\unity=\unity\cotensor_\unity\unity\xlongrightarrow{\varepsilon}K\cotensor_CL
\]
belongs to $\cat{A}_{\ge r}$.
\end{lemma}
\begin{proof}
Similar to Lemma \ref{lem:cotensor-bounded-below}.
\end{proof}

\begin{corollary}\label{cor:cotensors-preserve-positivity}
Let $k\ge 0$ and $m$ be integers such that $m\ge k+1$.
In Lemma, if $C$ is a copositive augmented $E_m$-coalgebra, and $K$
and $L$ are further $E_k$-coalgebras in $(\Mod_C)_{\unity/}$ which are
copositive as augmented $E_k$-coalgebras in $\cat{A}$, then
$K\cotensor_CL$ is copositive as an augmented $E_k$-coalgebra in
$\cat{A}$.
\end{corollary}

It follows that $n-1$ monoidal structures on
$\Bimod_{C\naga D}(\cat{A})_{>-\infty}$, all of which are given by the
cotensor product over $C^\op\tensor D$, has the compatibility required
for them to together define an $E_{n-1}$-monoidal structure on this
category of bounded below bimodules.

Thus, we can try to see if the construction of the Morita category
can be applied for restricted class of augmented coalgebras in
$\Bimod_{C\naga D}(\cat{A})_{>-\infty}$, to give an $n$-category
$\Coalg^+_{n-1}(\Bimod_{C\naga D}(\cat{A}))$.
This step will be similar to the argument we shall now give to observe
that the construction of $\Coalg^+_n(\cat{A})$ can be done assuming
that the construction of $\Coalg^+_{n-1}(\Bimod_{C\naga D}(\cat{A}))$
could be done.
Namely, what we shall do next is essentially an inductive step, which
will close our argument.
Let us do this now.

We would like to see that cotensor product operations between
categories of the form $\Coalg^+_{n-1}(\Bimod_{C\naga D}(\cat{A}))$
for copositive augmented $E_n$-coalgebras $C$, $D$, define
composition in the desired category
$\Coalg^\positive_n(\cat{A})$ enriched in $n$-categories.
Cotensor product of copositive objects remain copositive by
Corollary \ref{cor:cotensors-preserve-positivity}.
The functoriality of the cotensor operations follows from
Lemma \ref{lem:cotensors-compatible} (in the case $C_i$ are $C$).
Finally, the associativity of the composition defined by cotensor
product follows from Proposition \ref{lem:cotensor-associative}.

To summarize, the usual construction of the Morita category (as
outlined in \cite{tft}) works under our assumptions on $\cat{A}$ and
the copositivity of the class of objects we include, since in
construction of any composition of morphisms, application of the
totalization functor to any iterated multicosimplicial bar
construction which appear can be always postponed to the last step.

\subsubsection{}

In the next theorem, we shall see that the Koszul duality construction
is functorial on the positive Morita category, and gives an equivalence
of the algebraic and coalgebraic Morita categories.

Let us assume that the monoidal multiplication functor on $\cat{A}$
preserves geometric realization, and denote by
$\Alg^\positive_n(\cat{A})$ the positive part of the augmented version
of the Morita ($n+1$)-category of (augmented) $E_n$-algebras in
$\cat{A}$.
\begin{theorem}\label{thm:koszul-in-morita}
Let $\cat{A}$ be a symmetric monoidal complete soundly filtered
stable category with uniformly bounded sequential limits.
Assume that the monoidal multiplication functor on $\cat{A}$
preserves geometric realization.

Then for every $n$, the construction of the Koszul dual define a
symmetric monoidal functor
\[
\blank^!\colon\Alg^\positive_n(\cat{A})\longto\Coalg^\positive_n(\cat{A}).
\]
It is an equivalence with inverse given by the Koszul duality
construction.
\end{theorem}
\begin{proof}
\proofsec
Let us first describe the functor underlying the claimed symmetric
monoidal functor.
In order to do this, it suffices to consider the following,
more general case.
Namely, let $A_i$, $i=0,1$, be positive augmented $E_{n+1}$-algebras.
Then we would like to see that the Koszul duality constructions define
a functor
\begin{equation}\label{eq:morita-koszul-general}
\Alg^\positive_n\bigl(\Bimod_{A_0\naga
  A_1}\bigr)\longto\Coalg^\positive_n\bigl(\Bimod_{A_0^!\naga A_1^!}\bigr).
\end{equation}
Indeed, the original case is when $A_i$ are the unit algebra in
$\cat{A}$.

Similarly to how it was in the construction of the higher Morita
category, we need to consider here algebras $A_i$ possibly in the
category of bimodules over some $E_{n+2}$-algebras.
In order to understand \eqref{eq:morita-koszul-general} including this
case, recall first that in general, an $E_{k+1}$-algebra can be
considered as an $E_k$-algebra in the category of $E_1$-algebras.
Given an $E_{k+1}$-algebra $A$, let us denote by $A^{(!,1)}$ the
$E_1$-algebra in $E_k$-\emph{co}algebras which is obtained as the
$E_k$-Koszul dual of $A$.
If $A_i$, $i=0,1$, are $E_{k+1}$-algebras (possibly again in a
bimodule category, and inductively), and $B$ is an augmented
$E_k$-algebra in $(\Bimod_{A_0\naga A_1})_{>-\infty}$, then by $B^!$,
we mean the canonical augmented $E_k$-coalgebra in
$\bigl(\Bimod_{A_0^{(!,1)}\naga A_1^{(!,1)}}\bigr)_{>-\infty}$
(bimodules with respect to the $E_1$-\emph{algebra} structures of
$A_i^{(!,1)}$) lifting the $E_k$-Koszul dual (in inductively the
similar sense) of the augmented $E_k$-algebra underlying $B$ after
forgetting the bimodule structure of $B$ over $A_0$ and $A_1$.
Note that the $E_k$-monoidal structure of
$\bigl(\Bimod_{A_0^{(!,1)}\naga A_1^{(!,1)}}\bigr)_{>-\infty}$ is the
`plain' tensor product, lifting the $E_k$-monoidal structure
(underlying the $E_{k+1}$-monoidal structure) of the underlying
objects.
Note that if $A_i$ here are again algebras in a bimodule category,
then $A_i^{(!,1)}$ are interpreted in the similar way, and
inductively.
In particular, if one forgets all the way down to $\cat{A}$, then as
an augmented $E_k$-coalgebra in $\cat{A}$, $B^!$ is the Koszul dual of
the augmented $E_k$-algebra in $\cat{A}$ underlying $B$.
We are just taking into account the natural algebraic structures
carried by it.

\proofsec
Let us now describe the construction of
\eqref{eq:morita-koszul-general}.
Note that $A^!=(A^{(!,1)})^{(1,!)}$, where $\blank^{(1,!)}$ is the
Koszul duality construction with respect to the remaining
$E_1$-algebra structure.
Using this, \eqref{eq:morita-koszul-general} will be constructed as
the composition of two functors.
Namely, it will be constructed as a functor factoring through
$\Coalg^\positive_n\bigl(\Bimod_{A_0^{(!,1)}\naga A_1^{(!,1)}}\bigr)$.

\proofsec
The functor
\[
\Coalg^\positive_n\bigl(\Bimod_{A_0^{(!,1)}\naga
  A_1^{(!,1)}}\bigr)\longto\Coalg^\positive_n\bigl(\Bimod_{A_0^!\naga
  A_1^!}\bigr)
\]
to be one of the factors, will be induced from an op-lax
$E_n$-monoidal functor $\bigl(\Bimod_{A_0^{(!,1)}\naga
  A_1^{(!,1)}}\bigr)_{>-\infty}\to\bigl(\Bimod_{A_0^!\naga
  A_1^!}\bigr)_{>-\infty}$ whose underlying functor is
$\Dual_{A_0^{(!,1)}\naga A_1^{(!,1)}}\colon
K\mapsto\unity\tensor_{A_0^{(!,1)}}K\tensor_{A_0^{(!,1)}}\unity$.
Note that this functor will preserve copositivity of coalgebras once
it is given an $E_n$-monoidal structure.

To see the op-lax $E_n$-monoidal structure of
$\Dual:=\Dual_{A_0^{(!,1)}\naga A_1^{(!,1)}}$, let $S$ be a finite
set, and let $m$ be an $S$-ary operation in the operad $E_n$.
Then for a family $K=(K_s)_{s\in S}$ of objects of
$\bigl(\Bimod_{A_0^{(!,1)}\naga A_1^{(!,1)}}\bigr)_{>-\infty}$, we have
\[
\Dual m_!K\longto\Delta_m^*\Dual_{A_0^{(!,1)\,\tensor m}\naga
  A_1^{(!,1)\,\tensor m}}\widebar{m}_!K=m_*\Dual K,
\]
where
\begin{itemize}
\item
  $m_!\colon\bigl(\Bimod_{A_0^{(!,1)}\naga
    A_1^{(!,1)}}\bigr)^S\to\Bimod_{A_0^{(!,1)}\naga A_1^{(!,1)}}$
  is the monoidal multiplication along $m$,
\item $\Delta_m^*\colon\Bimod_{A_0^{!\,\tensor m}\naga A_1^{!\,\tensor
      m}}\to\Bimod_{A_0^!\naga A_1^!}$ is the (``co''-)extension of
  scalars along the comultiplication operations $\Delta_m$ along $m$
  of $A_0^!$ and $A_1^!$,
\item
  $\widebar{m}_!\colon\bigl(\Bimod_{A_0^{(!,1)}\naga
    A_1^{(!,1)}}\bigr)^S\to\Bimod_{A_0^{(!,1)\,\tensor m}\naga
    A_1^{(!,1)\,\tensor m}}$ is the external monoidal multiplication
  along $m$ (so $m_!=\Delta_m^*\widebar{m}_!$),
\item
  $m_*\colon\bigl(\Bimod_{A_0^!\naga A_1^!}\bigr)^S\to\Bimod_{A_0^!\naga
    A_1^!}$ is the monoidal multiplication along $m$,
\end{itemize}
and the map is the instance for $\widebar{m}_!K$ of the extension of
scalars of the $A_0^{!\,\tensor m}$--$A_1^{!\,\tensor m}$-bimodule map
$\Delta_{m*}\Dual\Delta_m^*\to\Dual_{A_0^{(!,1)\,\tensor m}\naga
  A_1^{(!,1)\,\tensor m}}$ induced from $\Delta_m$ of $A_0^{(!,1)}$
and $A_1^{(!,1)}$.

\proofsec
Next, we would like to describe the other factor
\begin{equation}\label{eq:morita-koszul-preparation}
\Alg^\positive_n\bigl(\Bimod_{A_0\naga
  A_1}\bigr)\longto\Coalg^\positive_n\bigl(\Bimod_{A_0^{(!,1)}\naga
  A_1^{(!,1)}}\bigr).
\end{equation}
If $B$ is an object of the source, then the object of
$\Coalg^\positive_n\bigl(\Bimod_{A_0^{(!,1)}\naga A_1^{(!,1)}}\bigr)$
associated to it is the $E_n$-Koszul dual $B^!$.
To see the functoriality of this construction, let $B_i$,
$i=0,1$, be objects of $\Alg^\positive_n(\Bimod_{A_0\naga A_1})$.
Then we first need a functor
\begin{equation}\label{eq:morita-koszul-on-map}
\Alg^\positive_{n-1}\bigl(\Bimod_{B_0\naga
  B_1}\bigr)\longto\Coalg^\positive_{n-1}\bigl(\Bimod_{B_0^!\naga
  B_1^!}\bigr).
\end{equation}
Note that this is the same form of functor as
\eqref{eq:morita-koszul-general}.
Therefore, we may assume that we have this functor by assuming we have
\eqref{eq:morita-koszul-preparation} for $n-1$ by an inductive
hypothesis, once we check the base case.
However, the base case is the identity functor of
$(\Bimod_{B_0\naga B_1})_{\ge 1}$ for positive $E_1$-algebras $B_i$
(in the bimodule category in the bimodule category in ...).

Next, we would like to see the compatibility of the functors
\eqref{eq:morita-koszul-on-map} with the compositions.
Thus, let $B_2$ be another object, and let maps
\[
B_0\xlongrightarrow{K_{01}}B_1\xrightarrow{K_{12}}B_2
\]
be given in $\Alg^\positive_n(\Bimod_{A_0\naga A_1})$.
Then the version of Lemma \ref{lem:cotensors-compatible} for positive
algebras implies that the $E_{n-1}$-Koszul dual
$(K_{01}\tensor_{B_1}K_{12})^!$ is equivalent by the canonical map to
the realization of a bicosimiplicial object which is also equivalent
to $K_{01}^!\tensor_{B_1^{(!,1)}}K_{12}^!$ by the canonical map, again
by Lemma \ref{lem:cotensors-compatible}.
Moreover, the canonical map
\begin{equation*}
\Dual_{B_0^{(!,1)}\naga
  B_2^{(!,1)}}\bigl(K_{01}^!\tensor_{B_1^{(!,1)}}K_{12}^!\bigr)\longto\bigl(\Dual_{B_0^{(!,1)}\naga
    B_1^{(!,1)}}K_{01}^!\bigr)\cotensor_{(B_1^{(!,1)})_{\phantom{1}}^{(1,!)}}\bigl(\Dual_{B_1^{(!,1)}\naga
    B_2^{(!,1)}}K_{12}^!\bigr),
\end{equation*}
is an equivalence by
Proposition \ref{prop:commuting-tensor-and-cotensor} and
Theorem \ref{thm:josh-complete-implies-koszul-complete}.

\proofsec
This essentially completes the inductive step, so we have given a
description of the underlying functor of the desired symmetric
monoidal functor.
Moreover, the symmetric monoidality of the functor is straightforward.

It follows in the same way that we also have a functor in the other
direction, and it follows from
Theorems \ref{thm:josh-complete-implies-koszul-complete} and
Lemma \ref{lem:koszul-complete-coalgebra} that these are inverse to
each other.
\end{proof}

\begin{remark}
As the proof shows, the equivalence is in fact more than an
equivalence of ($n+1$)-categories.
Namely, the equivalence $A\equivwith A^{!!}$ for $A$ in any dimension
is an honest equivalence of algebras, rather than merely an
equivalence in the Morita category.
\end{remark}

\begin{remark}\label{rem:dropping-preservation-of-realization}
Theorem seems to be suggesting that $\Coalg^\positive_n(\cat{A})$ is a
meaningful thing at least in the case where the monoidal operation of
$\cat{A}$ preserves geometric realizations.
However, the construction of $\Coalg^\positive_n(\cat{A})$ was
independent of this assumption, and a similar construction for
$\Alg^\positive_n(\cat{A})$ works without preservation of
geometric realizations.
Moreover, Theorem remains true in this generality.
\end{remark}

Recall that any $E_n$-algebra $A$, as an object of the Morita
($n+1$)-category $\Alg_n(\cat{A})$, is $n$-dualizable.
All dualizability data are in fact given by $A$, considered as
suitable morphisms in $\Alg_n(\cat{A})$.

It is then immediate to see that if $A$ is an augmented
$E_n$-algebra, then the dualizability data (and the field theory) for
$A$ in $\Alg_n(\cat{A})$ can be lifted to those for $A$ in
$\Alg^*_n(\cat{A})$, the augmented version of $\Alg_n(\cat{A})$.
Moreover, if $A$ is positive, then those data belongs to
$\Alg^\positive_n(\cat{A})$.
In particular, $A$ will be $n$-dualizable in
$\Alg^\positive_n(\cat{A})$.

\begin{corollary}
Let $\cat{A}$ be a symmetric monoidal complete soundly filtered
stable category with uniformly bounded sequential limits.
Then any object of the symmetric monoidal category
$\Coalg^\positive_n(\cat{A})$ is $n$-dualizable.
\end{corollary}

There is a concrete description of the fully extended $n$-dimensional
framed topological field theory associated to an object
$A\in\Alg_n(\cat{A})$, using the topological chiral homology.
See Lurie \cite{tft}.
In \cite{poincare}, we shall give a concrete description of the
framed topological field theory associated to a copositive
$E_n$-coalgebra, using \emph{compactly supported} topological chiral
homology.
See also Francis \cite{francis} for an earlier, and closely related
result.
Specifically, we use the Poincar\'e type duality theorem on the
compactly supported topological chiral homology, analogous to Lurie's
``non-abelian'' Poincar\'e duality theorem \cite{higher-alg}.

\begin{remark}
The key for all the results of this section was good control of the
behaviour of the limits and colimits with respect to the monoidal
structure.
Another symmetric monoidal category in which both the limits and
colimits behave well is the Cartesian symmetric monoidal category of
spaces.
This is the context in which Lurie considers his generalization of the
Poincar\'e duality theorem.

The coalgebraic higher Morita category in a Cartesian symmetric
monoidal category (which is closed under the finite limits) was
identified by Ben-Zvi and Nadler with the ($n+1$)-category of iterated
correspondences \cite[Remark 1.17]{bzn}.
The Koszul duality in the category of spaces is given by the iterated
looping and delooping constructions.
Suitable analogues of our results hold in this context, and are
consequences of the iterated loop space theory.
\end{remark}

\end{document}